\documentclass[12pt]{article}
\usepackage{amsfonts}
\usepackage{epsfig}
\usepackage{graphics}
\usepackage{inputenc}
\usepackage{lineno,hyperref}
\usepackage{amsmath, amssymb,amsthm, mathtools, empheq}
\usepackage{paralist, mathtools}
\usepackage{caption}
\usepackage{subcaption}
\usepackage{ulem}
\usepackage{placeins}
\usepackage{enumerate, comment}
\usepackage{mathrsfs}
\usepackage{mathtools}
\usepackage{xcolor}

\usepackage{authblk}

\usepackage[T2A]{fontenc}
\def \ve{\varepsilon}

\def \div {{\rm div}}

\numberwithin{equation}{section}

\newtheorem{thm}{Theorem}%[section]
\numberwithin{thm}{section}
\newtheorem{rem}[thm]{Remark}
\newtheorem{prop}[thm]{Proposition}
\newtheorem{lem}[thm]{Lemma}

\title{Homogenization of nonlocal spectral problems}

\author[1]{Andrey  Piatnitski}
\author[2]{Volodymyr Rybalko}
\affil[1]{The Arctic University of Norway, UiT, campus Narvik, Norway, 
e-mail:  apiatnitski@gmail.com, andrey.piatnitski@uit.no}
\affil[2]{B. Verkin Institute for Low Temperature Physics and Engineering of the National Academy of Sciences of Ukraine, Ukraine;
Department of Mathematical Sciences, Chalmers University of Technology and Gothenburg University, Sweden,
e-mail: vrybalko@ilt.kharkov.ua, rybalko@chalmers.se}

\begin{document}
\maketitle

\begin{abstract} We study asymptotic behavior of the bottom point of 
the spectrum of convolution type operators in  environments with locally periodic microstructure. 
We show that its limit is described by an additive eigenvalue problem for Hamilton-Jacobi equation.
In the periodic case we establish a more accurate two-term asymptotic formula. 
\end{abstract}'

%MSC: 47A10, 60J85, 47D06


\section{Introduction}

This work deals with homogenization of spectral problems for nonlocal convolution type
operators in environments with locally periodic microstructure. In a bounded $C^1$-domain $\Omega\subset\mathbb R^d$ we consider the spectral
problem
\begin{equation}
\mathcal{L}_\ve\rho_\ve = \lambda_\ve \rho_\ve(x)\quad \text{in}\ \Omega
\label{start_eq}
\end{equation}
for the operator
\begin{equation}
\mathcal{L}_\ve\rho =-\frac{1}{\ve^{d}}\int_\Omega J\Bigl({\frac{x-y}{\ve}}\Bigr)%\kappa_\ve(x,y)
\kappa\Bigl(x,y, \frac{x}{\ve},\frac{y}{\ve}\Bigr)
\rho(y)dy+
%\frac{1}{\ve^{2}}
a\Bigl(x, \frac{x}{\ve}\Bigr)\rho(x),
%a_\ve(x)\rho(x)
\label{main_operator_locper}
\end{equation}
with a small parameter $\ve>0$ that characterizes the microscopic length scale of the medium.
Under natural positiveness and periodicity conditions as well as fast decay of $J$ at infinity we study the bottom of the spectrum of this problem. The main 
focus of this work is on the asymptotic behavior as $\ve\to 0$ of the point of the spectrum with the smallest real part.

Many models in mathematical biology and population dynamics take into account nonlocal interactions in the studied
systems.  These interactions are described by convolution-type
integral operators with integrable kernels. More specifically, a simplest nonlocal model of population dynamics reads (see, e.g.,  \cite{HutMarMisVic2003}), \cite{OthDunAlt1988})
\begin{equation*}
\partial_t\rho(x,t)-\int_\Omega J(x-y) \rho(y,t)dy+\int_{\mathbb{R}^d} J(y-x) dy\, \rho(x,t)=0 \quad \text{in}\ \ \Omega,\quad \rho =0\quad \text{on}\ \mathbb{R}^d\setminus \Omega,
\end{equation*}
where $\rho$ denotes  population density,  $J(x-y)\geq 0$ is a dispersal kernel that describes the rate of jumps from the location $y$ to the location $x$ and the above equation
defines nonlocal  transport in $\Omega$, while the Dirichlet condition $\rho=0$ (imposed everywhere on $\mathbb{R}^d\setminus \Omega$) represents
the case of a hostile exterior domain. In the case when the growth of the population is taken into account the above equation is also
supplemented by an additional KPP type (local) nonlinear term (see, e.g., \cite{Fif2017}, \cite{GriHinHutMisVic2005}). Then the large time behavior of $\rho(x,t)$ can be qualitatively characterized by
linearizing the problem and studying the bottom part of the spectrum of the corresponding operator, in particular, an optimal persistence criterion
is formulated in terms of the bottom point of the spectrum \cite{BatZha2007}, \cite{Cov2010}, \cite{Cov2015}, \cite{GarRos2009}, \cite{KaoLouShe2010}, \cite{SheZha2012}.

%\textcolor{blue} {In the case of strongly inhomogeneous environments this amounts to the spectral analysis for nonlocal operators of the form \eqref{main_operator_locper}. }
%\textcolor{green}{this sentence is out of place}

To model nonlocal diffusion in strongly inhomogeneous media, in \cite{PiaZhi2017}, \cite{PiaZhi2019} evolution problems with dispersal kernels of the form $J(x-y)\kappa(x,y)$
with integrable $J\geq 0$ %\sout{(satisfying natural moment condition)}
and positive periodic $\kappa$ were considered in the parabolic scaling $t\to t/\ve^2$, $x\to x/\ve$. 
It was shown in  \cite{PiaZhi2017}  that  the asymptotic behavior of solutions is described by a local effective diffusion problem in the symmetric case (when  $J(z)=J(-z)$ and $\kappa(x,y)=\kappa(y,x)$),
while \cite{PiaZhi2019} revealed a large effective drift  %\sout{(of the order $1/\ve$)} 
appearing in asymmetric case, and  the corresponding homogenization result was established  in rapidly moving coordinates.  The approach
in \cite{PiaZhi2019}  was developed  for problems stated in the whole space $\mathbb{R}^d$, and it fails to work in the case of a bounded domain because of the presence of large effective drift. To overcome this
difficulty one can combine the study of the evolution problem with the spectral analysis of problem \eqref{start_eq}.
%It is relevant then to study the spectral problem

%\textcolor{blue}{ our problem is singularly perturbed ...  }

%It should be noted that the symmetric part of the convolution operator
%\begin{equation}
%\mathcal{L}_\ve\rho =-\frac{1}{\ve^{d+1}}\int_{\mathbb R^d} J\Bigl({\frac{x-y}{\ve}}\Bigr)%\kappa_\ve(x,y)
%%\kappa\Bigl(\frac{x}{\ve},\frac{y}{\ve}\Bigr)
%(\rho(y)-\rho(x))dy
%%+ \frac{1}{\ve^{2}}
%%a\Bigl(x, \frac{x}{\ve}\Bigr)\rho(x),
%%a_\ve(x)\rho(x)
%\label{operator_locper_generator_conv}
%\end{equation}
%can be considered as a small singular perturbation if its anti-symmetric part.
%Indeed, letting
%$$
%J^{\rm sym}(\xi,\eta)=\frac12\big(J(\xi-\eta)+J(\eta-\xi)\big),\qquad
%J^{\rm as}(\xi,\eta)=\frac12\big(J(\xi-\eta)-J(\eta-\xi)\big),
%$$
%%$$
%%K^{\rm sym}(\xi,\eta)=\frac12\big(a(\xi-\eta)\kappa(\xi,\eta)+a(\eta-\xi)\kappa(\eta,\xi)\big),\qquad
%%K^{\rm as}(\xi,\eta)=\frac12\big(a(\xi-\eta)\kappa(\xi,\eta)-a(\eta-\xi)\kappa(\eta,\xi)\big),
%%$$
%one obtains for any $\rho\in C_0^\infty(\mathbb R^d)$ the following relations:
%$$
%\frac{1}{\ve^{d+1}} \int_{\mathbb R^d} J^{\rm sym}\Bigl(\frac{x-y}{\ve}\Bigr)
%(\rho(y)-\rho(x))dy \,=\,  \ve \mathrm{div}\big(J^{\rm hom}\nabla\rho\big)+o(\ve) ,
%$$
%and
%$$
%\frac{1}{\ve^{d+1}} \int_{\mathbb R^d} J^{\rm as}\Bigl(\frac{x-y}{\ve}\Bigr)
%(\rho(y)-\rho(x))dy  \,=\,    \bar b\nabla\rho+O(\ve) ;
%$$
%here $J^{\rm hom}_{ij}=\int_{\mathbb R^d} z_iz_j J^{\rm sym}(z)dz$ and $\bar b_{j}=\int_{\mathbb R^d} z_j J^{\rm as}(z)dz$.
%The derivation of these relations can be found in \cite{PiaZhi19}, see also Section \ref{}.

%{\color{red}
Peculiar features of the spectral problem  \eqref{start_eq} can be seen in the particular case of 
nonlocal diffusion operator in $\mathbb{R}^d$
\begin{equation}
\mathcal{L}_\ve^D\rho =-\frac{1}{\ve^{d}}\int_{\mathbb R^{d+1}} J\Bigl({\frac{x-y}{\ve}}\Bigr)%\kappa_\ve(x,y)
\kappa\Bigl(\frac{x}{\ve},\frac{y}{\ve}\Bigr)
(\rho(y)-\rho(x))dy
%+ \frac{1}{\ve^{2}}
%a\Bigl(x, \frac{x}{\ve}\Bigr)\rho(x),
%a_\ve(x)\rho(x)
\label{operator_locper_ctype}
\end{equation}
with a periodic coefficient  $\kappa$. In light of \cite{PiaZhi2019} the operator $\frac{1}{\ve} \mathcal{L}_\ve^D$ can be regarded as a small singular perturbation %of order $O(\ve^2)$ 
of a first order differential operator $B\cdot\nabla$ with constant drift vector $B$  when $\ve$ is small. Moreover, $\frac{1}{\ve^2} \mathcal{L}_\ve^D$ is close, in the sense of the 
resolvent convergence, 
to the operator  $-\,\div (A\nabla\, \cdot\,) + \frac{1}{\ve}B \cdot\nabla$ with 
some constant matrix $A>0$. Therefore, one can 
expect that the asymptotic behavior of the bottom part of 
the spectrum of the operator $\mathcal{L}_\ve$ (given by \eqref{main_operator_locper})
is also somehow similar to that of  singularly perturbed elliptic differential operators.

%nonlocal diffusion opera
%
%In light of \cite{PiaZhi2019}, a nonlocal 
%convolution type operator  
%\begin{equation}
%\mathcal{L}_\ve^D\rho =-\frac{1}{\ve^{d+1}}\int_{\mathbb R^{d+1}} J\Bigl({\frac{x-y}{\ve}}\Bigr)%\kappa_\ve(x,y)
%\kappa\Bigl(\frac{x}{\ve},\frac{y}{\ve}\Bigr)
%(\rho(y)-\rho(x))dy
%%+ \frac{1}{\ve^{2}}
%%a\Bigl(x, \frac{x}{\ve}\Bigr)\rho(x),
%%a_\ve(x)\rho(x)
%\label{operator_locper_ctype}
%\end{equation}
%with a periodic coefficient  $\kappa$ can be regarded as a small singular perturbation %of order $O(\ve^2)$ 
%of a first order differential operator $B\cdot\nabla$ when $\ve$ is small. Moreover,  
%this perturbation is given, to the leading terms, by the operator %a singularly perturbed differential operator of the form 
%$-\ve \,\div (A\nabla\, \cdot\,) + B \cdot\nabla$ with 
%some constant matrix $A>0$ and vector $B$. Therefore, one can 
%expect that the asymptotic behavior of the bottom part of 
%the spectrum of the operator $\mathcal{L}_\ve$ 
%is also somehow similar to that of  singularly perturbed elliptic differential operators. This analogy 
%also suggests that  technically, locally periodic case differs drastically from the periodic one.  

It should be noted that, in contrast with differential operators,
the point of the spectrum with the smallest real part need not be a principal eigenvalue, it might lie on the edge of the essential spectrum. However, 
we show 
that  the limit of this point as $\ve\to 0$ can be specified in terms of an additive eigenvalue for an effective
Hamilton-Jacobi equation,  like in the case of   principal eigenvalues of singularly perturbed elliptic differential operators.  Despite this
similarity of  results, related to the fact that the operator $\mathcal{L}_\ve$ is becoming more and more localized for small $\ve$, many technical aspect 
in proofs  in local and nonlocal cases are different. 
%,
%see the viscosity solutions "users guide"  \cite{CraIshLio1992} for the further details on additive eigenvalue problems.
%}

%\textcolor{green}{
%Another difficulty is related to the fact that the Harnack or Bernstein estimates are not known for  solutions or almost
%solutions of  \eqref{start_eq}--\eqref{main_operator_locper}.  To make our technique work we apply  smoothing operators
%in the scale of order $\eps$.
%}\includegraphics[]{SpectrNonlocFinal2024.pdf}

The asymptotic behavior of the principal eigenpair
of singularly perturbed local (differential) convection-diffusion operators with oscillating coefficients was studied in \cite{PiaRyb2016} (see also \cite{PiaRybRyb2014}), where 
the Cole–Hopf transformation $e^{-W_\ve(x)/\ve}$ of the first eigenfunction was used, yielding
a perturbed Hamilton-Jacobi equation.  The latter was naturally dealt with by means of the vanishing viscosity techniques.
In the case of nonlocal operators the Cole–Hopf transformation does not lead to a perturbed Hamilton-Jacobi equation, 
instead we exploit monotonicity of the map $W_\ve\mapsto e^{-W_\ve/\ve}$ to devise a version of perturbed test functions
method \cite{Eva1989} for nonlocal opearators. The main difficulty however is in finding relevant uniform
bounds for the function $W_\ve$. In the case of local operators appropriate tools are the Berstain estimates and the
Harnack inequality.  The Berstain's method is developed for local operators and it is not clear how to adapt it in the nonlocal
case, while the known results \cite{Cov2012} on the Harnack type inequalities are established for dispersal kernels with a finite support and
are not scaling invariant. We prove instead uniform estimates in Lemma \ref{HighTechLemma} to establish the existence
of half-relaxed limits of $W_\ve$.
%For nonlocal operatorsto find a priory estimates
%
%
% and we deal with a half-relaxed limit of $W_\ve$ to obtain a lower
% bound for

When the function $\kappa$ in \eqref{main_operator_locper}  does not depend on the slow variables (periodic case)
the solution of the effective Hamilton-Jacobi problem is a linear function. In this case, using a factorization trick 
%\textcolor{red}{
inspired by asymptotic analysis of differential operators with periodic coefficients  \cite{Capdeb1998},  we establish more accurate  
two-term asymptotic formula for eigenvalues (if exist) in the bottom part of the spectrum. This, in particular, generalizes an asymptotic
result of \cite{BerCovHoa2016} obtained for symmetric operators with homogeneous dispersal kernels (remark however 
that in this special case the analysis in \cite{BerCovHoa2016} covers also unbounded domains that are outside of the scope of the present work).

It is known that operator \eqref{operator_locper_ctype} is the generator of a jump Markov process in $\mathbb R^d$. Therefore, one of the way of
obtaining homogenization results for problems involving operator \eqref{operator_locper_ctype} is based on probabilistic interpretation
of solutions to these problems and on the limit theorems for the corresponding jump Markov processes. 
In particular, an alternative approach of studying the principal eigenpair of %the operator $\mathcal{L}_\ve$ in
problem \eqref{start_eq}--\eqref{main_operator_locper} could rely on a probabilistic interpretation 
of  operator \eqref{main_operator_locper}. %$\mathcal{L}_\eps$. 
Since this operator is the generator of a jump Markov process with birth and death,
one can try to exploit the large deviation principle  for this process in order to investigate the asymptotic
properties of problem \eqref{start_eq}--\eqref{main_operator_locper}.  In the case of operators defined in  
\eqref{operator_locper_ctype} and  similar locally periodic operators 
 the large deviation
result for the mentioned jump processes was obtained in  \cite{PiaPirZhi2022}. 
%We strongly believe that the approach developed in
% \cite{PiaPirZhi2022} also applies to the case of a general locally periodic operator of the form \eqref{main_operator_locper}. This is the subject of the future work.

The paper is organized as follows. In Section \ref{Setup} we state Theorem \ref{Golovna_locperiodic}  describing
the limit  as $\ve\to 0$ of the bottom point of the spectrum of the operator  \eqref{main_operator_locper}
%$\mathcal{L}_\ve$ given
in general locally periodic case. Section \ref{Periodiccase} is devoted to
establishing more precise asymptotics of eigenvalues in the case of
periodic environments. Finally,
Section \ref{Locallyperiodiccase} contains the proof of Theorem  \ref{Golovna_locperiodic}.

%nonlocal
%diffusion reflecting long range interaction in a system.  It is described by
%and one naturally interested in studying questions

%Nonlocal diffusion is often
%Evolution processes in the models of mathematical biology and population dynamics are often
%described in terms of parabolic equation of the form ∂tu = −Au with a nonlocal convolution type
%operator A, the nonlocality of A reflects the fact that the interaction in these models is nonlocal.

%2) Homogenization of nonlocal operators (large drift)

%3) Spectral problems

%Viscosity solutions "users guide":  \cite{CraIshLio1992}

\section{Problem setup.Convergence result for the bottom point of the spectrum}
\label{Setup}

%We consider spectral problem \eqref{start_eq} in a bounded (open) domain $\Omega$ in $\mathbb{R}^d$ with $C^1$-boundary.
%The nonlocal operator $\mathcal{L}_\ve$ in the problem depends on the small parameter $\ve>0$ characterizing the microscopic scale of the medium.
%
%a nonlocal counterpart of a singularly perturbed
%elliptic operator with oscillating coefficient
%
%under the study is
%
%%\begin{equation}
%%\mathcal{L}_\ve\rho_\ve = \lambda_\ve \rho_\ve(x)\quad \text{in}\ \Omega,
%%\label{start_eq}
%%\end{equation}
%where $\Omega$ is a bounded domain  in $\mathbb{R}^d$ with $C^1$-boundary, and
%$\mathcal{L}_\ve$ is the nonlocal operator
%given by
%%in $L^2(\Omega)$ of the form
%\begin{equation}
%\mathcal{L}_\ve\rho =-\frac{1}{\ve^{d}}\int_\Omega J\Bigl({\frac{x-y}{\ve}}\Bigr)%\kappa_\ve(x,y)
%\kappa\Bigl(x,y, \frac{x}{\ve},\frac{y}{\ve}\Bigr)
%\rho(y)dy+
%%\frac{1}{\ve^{2}}
%a\Bigl(x, \frac{x}{\ve}\Bigr)\rho(x),
%%a_\ve(x)\rho(x)
%\label{main_operator_locper}
%\end{equation}
%$\ve>0$ being a small parameter characterizing the microscopic scale of the medium.
%%with fast oscillating functions $\kappa(x,y,\xi,\eta)$  x/\ve, y/\ve)$ and $a(x,x/\ve)$, periodic in
%with fast oscillating periodic or locally periodic functions $\kappa_\ve(x,y)$ and  $a_\ve(x)$, i.e. either $\kappa_\ve(x,y)=\kappa(x/\ve, y/\ve)$, $a_\ve(x)=a(x/\ve)$ or $\kappa_\ve(x,y)=\kappa(x,y, x/\ve, y/\ve)$, $a_\ve(x)=a(x,x/\ve)$.
We begin with specifying assumptions on the functions $J$, $\kappa$ and $a$ appearing in the definition \eqref{main_operator_locper} of  $\mathcal{L}_\ve$.
We assume that $J$ satisfies
%\sout{We assume that the functions $J$, $\kappa$ and $a$ appearing in the definition \eqref{main_operator_locper} of  $\mathcal{L}_\ve$ satisfy the following conditions,}
\begin{equation}
	\label{umova_dva}
	%\kappa>0,\quad
	J\in C(\mathbb{R}^d), \ J(0)>0 \ \ \text{ and}\ \
	0\leq J(z)\leq C e^{ %\textcolor{red}{\theta}
	-|z|^{1+\beta} }\ \forall z\in \mathbb{R}^d,
	%$$
	%J(0)>0,\quad |J(z)|\leq Ce^{-|z|^\beta) \ \text{for some}\ \beta >1,
		%$$
	\end{equation}
for some $C, \beta>0$, %\textcolor{red}{and $\theta>0$}  ???,
while
 \begin{equation}
 \label{akappa}
  \kappa>0\quad \text{and}\quad \kappa\in C(\overline{\Omega}\times\overline{\Omega} \times \mathbb{T}^d\times \mathbb{T}^d),\quad  a\in C(\overline{\Omega}\times \mathbb{T}^d),
  \end{equation}
%\begin{equation}
%\label{umova_odyn}
%  a\in C(\mathbb{T}^d) \ \ \text{ and }
%%$$
%%J(0)>0,\quad |J(z)|\leq Ce^{-|z|^\beta) \ \text{for some}\ \beta >1,
%%$$
%\ \
%%\, a(x) \text{ are continuous periodic functions,}\ a\in C(\mathbb{T}^d),\,
%\kappa\in C(\mathbb{T}^d\times \mathbb{T}^d), \ \kappa(x,y)>0,
%\end{equation}
where $\mathbb{T}^d$ denotes the torus $\mathbb{T}^d=\mathbb{R}^d/\mathbb{Z}^d$ and we identify periodic functions in $\mathbb{R}^d$ with functions defined on $\mathbb{T}^d$.
%\textcolor{red}{
Thus coefficients $\kappa$ and $a$ in \eqref{main_operator_locper} are periodic functions of the fast variables $\xi=x/\ve$ and $\eta=y/\ve$, so that
operator \eqref{main_operator_locper} correspond to a locally periodic environment.

%\newpage

We are interested in the asymptotic behavior as $\ve\to 0$ of the following quantity introduced in \cite{Cov2010},
\begin{equation}
%\begin{aligned}
\lambda_\ve=\sup\Bigl\{\lambda \,\bigl| \bigr.%&
\, \exists v\in C(\overline\Omega),\, v>0 \ \text{such that} \
%&
%-\frac{1}{\ve^d}
%-\int_{\Omega}  E_\ve(x,y)%J\Bigl(\frac{x-y}{\ve}\Bigr)\kappa\Bigl(x,y,\frac{x}{\ve}, \frac{y}{\ve}\Bigr)
%\frac{v(y)}{v(x)}dy+
% a\Bigl(x,\frac{x}{\ve}\Bigr)
\mathcal{L}_\ve v
 \geq \lambda v \ \text{in}\ \Omega \Bigr\}.
% \end{aligned}
\label{Pr_e_v_supformula}
\end{equation}
It is known  \cite{LiCovWan2017} (Theorem 2.2) that $\lambda_\ve$ belongs to the spectrum  $\sigma(\mathcal{L}_\ve)$
of the operator $\mathcal{L}_\ve$ (considered in $L^2(\Omega)$ or $C(\overline\Omega)$) and
$\lambda_\ve =\inf\{{\rm Re}\lambda\,|\, \lambda\in \sigma(\mathcal{L}_\ve)\}$, i.e. $\lambda_\ve$ is the bottom
point of the spectrum.  Typically $\lambda_\ve$ is the principal eigenvalue of the operator
$\mathcal{L}_\ve$, i.e.  an isolated simple eigenvalue (with minimal real part)
whose corresponding eigenfunction can be chosen strictly positive,
as in the case of elliptic differential operators.
% \textcolor{red}{/local convention-diffusion operators}.
However if  $\lambda_\ve=\min_{\overline\Omega} a(x, \frac{x}{\ve})$
rather than  $\lambda_\ve<\min_{\overline\Omega} a(x, \frac{x}{\ve})$, the principal eigenvalue  does not exist
and $\lambda_\ve$ is the bottom point of the essential spectrum of  $\mathcal{L}_\ve$.
In both cases the sign of $\lambda_\ve$ is crucial for stability of the corresponding evolution semigroup $e^{-\mathcal{L}_\ve t}$
or for the maximum principle to hold, see Theorem 2.3 in \cite{LiCovWan2017}.

To state the  main result on the asymptotic behavior of $\lambda_\ve$  for locally periodic environments,
introduce the following  function
\begin{equation}
\begin{aligned}
H(p,x)=\sup\Biggl\{\lambda \,\bigl.\Bigr| \Biggr.&\, \exists \varphi\in C(\mathbb{T}^d),%_{\#}
\, \varphi>0, \ \text{such that} \\
&\Biggl.
-\int_{\mathbb{R}^d} J(z)e^{p\cdot z}\kappa(x,x,\xi,\xi-z) \varphi(\xi-z)d z+
 a(x, \xi) \varphi(\xi)\geq \lambda \varphi(\xi) \ \text{on}\ \mathbb{T}^d \Biggr\}.
 \end{aligned}
\label{Hamilt_all_p_sup_dep_on_x}
\end{equation}
\begin{thm}
\label{Golovna_locperiodic} %{main_operator_locper}
Suppose that $J$ satisfies
\eqref{umova_dva} and $\kappa$, $a$ satisfy \eqref{akappa}.
% $\kappa\in C(\Omega\times\Omega \times \mathbb{T}^d\times \mathbb{T}^d)$,  $\kappa>0$,
%$a\in C(\Omega\times \mathbb{T}^d)$ and
Then
\begin{equation}
\lambda_\ve\to -\Lambda \quad\text{as}\ \ve\to 0, %\quad \text{where}\ \Lambda = \inf_{V\in C^1(\overline{\Omega})}\max_{x\in\overline{\Omega}} - H(\nabla V(x),x),
 \label{NuTodiVseOk}
 \end{equation}
where $\Lambda$ is a unique additive eigenvalue of the problem
\begin{equation}
\label{HJ_lim_eq}
- H(\nabla W(x),x)=\Lambda\quad\text{in}\ \Omega,\quad - H(\nabla W(x),x)\geq \Lambda\quad \text{on}\  \partial\Omega.
\end{equation}
 \end{thm}
 \noindent
Both the equation and the boundary condition in \eqref{HJ_lim_eq} are understood in the viscosity sense, see, e.g.,  \cite{CraIshLio1992}.
 It follows from the definition \eqref{Hamilt_all_p_sup_dep_on_x} that function $H(p,x)$ is continuous, $H\in C(\mathbb{R}^d\times\overline{\Omega})$,
 concave in $p$ and $H(p,x)\to -\infty$ uniformly in $x\in \overline{\Omega}$
as $|p|\to \infty$. Then (see, e.g., \cite{Cap-DolLio1990})  there is a unique $\Lambda$
such that problem \eqref{HJ_lim_eq} has a continuous viscosity solution. Moreover,  the additive eigenvalue $\Lambda$  can be calculated by the following formula
\begin{equation}
\label{Lambda_pershaformulka}
\Lambda = \inf_{W\in C^1(\overline{\Omega})}\max_{x\in\overline{\Omega}} - H(\nabla W(x),x).
\end{equation}
Yet another representation for  $\Lambda$ is
given by minimization of action functional for the Lagrangian $L(q,x)=\max_{p\in\mathbb{R}^d} (q\cdot p+H(p,x))$,
\begin{equation}
\label{Diya}
\Lambda=-\inf\Biggl\{\frac{1}{T}\int_0^T L(\dot\xi(t),\xi(t))dt \Bigl.\Bigr|\ T>0,\xi \in W^{1,\infty}(0,T;\overline{\Omega})\Biggr\}.
\end{equation}
%\textcolor{red}{(compare the result obtained in this
%work with that obtained in \cite{PiaPirZhi2022}?)}

%Similar to Theorem \ref{Golovna_locperiodic} result  holds in the case of  singularly perturbed local convection-diffusion operator \cite{PiaRyb2016}, where the Cole–Hopf transformation $e^{-W_\ve(x)/\ve}$ of the first eigenfunction yields
%a perturbed Hamilton-Jacobi equation and it is naturally dealt with the help of the vanishing viscosity techniques.
%

%\begin{equation}
%\label{Diya}
%\Lambda=-\inf\Biggl\{\frac{1}{T}\int_0^T L(\dot\xi(t),\xi(t))dt \Bigl.\Bigr|\ T>0,\xi \in W^{1,\infty}(0,T;\overline{\Omega})\Biggr\}
%\end{equation}
%with $L(q,x)=\max_{p\in\mathbb{R}^d} (q\cdot p+H(p,x))$ and $H(p,x)$ given by
%

% \begin{rem} Equivalently: $-H(\nabla V(x),x)\leq \Lambda$ in $\Omega$, $-H(\nabla V(x),x)\geq\Lambda$ in  $\Omega$ and on $\partial\Omega$.
 %\end{rem}

	%\kappa>0,\quad a\in C(\Omega\times\Omega \times \mathbb{T}^d\times \mathbb{T}^d)$, $a\in C(\Omega\times \mathbb{T}^d)$,

In the periodic case $H(p,x)$ is independent of $x$, $H(p,x)=H(p)$, and one can show that the additive eigenvalue $\Lambda$ in this case is given by $\Lambda=-\max H(p)$ while solutions $W(x)$ of
\eqref{HJ_lim_eq} are linear functions $W(x)=p\cdot x$ with  $p$ solving $H(p)=-\Lambda$.  This suggests the representation \eqref{factor_u}  for eigenfunctions to be used in the study of periodic case.

\section{Periodic case}
\label{Periodiccase}

In this section we consider a particular case of problem \eqref{start_eq} when functions $\kappa$ and $a$ do not depend
on the slow variables $x$ and $y$, so that the operator $\mathcal{L}_\ve$ has the following form
\begin{equation}
\mathcal{L}_\ve\rho (x)=-\frac{1}{\ve^{d}}\int_\Omega J\Bigl({\frac{x-y}{\ve}}\Bigr)%\kappa_\ve(x,y)
\kappa\Bigl(\frac{x}{\ve},\frac{y}{\ve}\Bigr)
\rho(y)dy+
%\frac{1}{\ve^{2}}
a\Bigl(\frac{x}{\ve}\Bigr)\rho(x)
%a_\ve(x)\rho(x)
\label{main_operator}
\end{equation}
with
\begin{equation}
\label{umova_odyn}
\kappa\in C(\mathbb{T}^d\times \mathbb{T}^d),\ \kappa>0,\ a\in C(\mathbb{T}^d),
\end{equation}
and $J$ satisfying \eqref{umova_dva}.

Let us introduce the notation
\begin{equation}
m=\min a(x), \quad M=\max a(x),
\label{minmax_a}
\end{equation}
and for any $p\in\mathbb{R}^d$ define

\begin{equation}
\begin{aligned}
H(p)=\sup\Biggl\{\lambda \,\Bigl.\Bigr| \Biggr.&\, \exists \varphi\in C(\mathbb{T}^d),%_{\#}
\, \varphi>0, \ \text{such that} \\
&\Biggl.
-\int_{\mathbb{R}^d} J(\xi-\eta)e^{p\cdot (\xi-\eta)}\kappa(\xi,\eta) \varphi(\eta)d\eta+
 a(\xi) \varphi(\xi)\geq \lambda \varphi(\xi) \ \text{on}\ \mathbb{T}^d \Biggr\}.
 \end{aligned}
\label{Hamilt_all_p_sup}
\end{equation}
%\textcolor{blue}{? $H(p)$ can be equivalently defined (see, e.g., \cite{LiCovWan2017}) as
%	\begin{equation}
%	\begin{aligned}
%		H(p)=\inf\Biggl\{\lambda \,\bigl.\Bigr| \Biggr.&\, \exists \varphi\in C_{\#}(\mathbb{T}^d),\, \varphi>0, \ \text{such that} \\
%		&\Biggl.
%		-\int_{\mathbb{R}^d} J(\xi-\eta)e^{p\cdot (\xi-\eta)}\kappa(\xi,\eta) \varphi(\eta)d\eta+
%		a(\xi) \varphi(\xi)\leq \lambda \varphi(\xi) \ \text{on}\ \mathbb{T}^d \Biggr\}.
%	\end{aligned}
%	\label{Hamilt_all_p_inf}
%\end{equation}}
It follows from \eqref{Hamilt_all_p_sup} that $H(p)$ is a continuous concave function, taking finite values for all $p\in \mathbb{R}^d$ and  such that
$H(p)\to -\infty$ as $|p|\to \infty$.
Also, by Theorem 2.2 in \cite{LiCovWan2017} one has $H(p)\leq m$.
Let $p_0$ be a maximum point of $H(p)$,
\begin{equation}
\label{maximum_tochka}
H(p_0)=\max H(p).
\end{equation}

The asymptotic behavior of the bottom part of the spectrum of the operator
$\mathcal{L}_\ve$ in the periodic case is described in the following

\begin{thm}
\label{Golovna_periodic}
Assume that conditions \eqref{umova_dva}, \eqref{umova_odyn} are fulfilled.
Let $\lambda_\ve$ be the point of the spectrum of  $\mathcal{L}_\ve$ with the minimal real part.
Then  $\lambda_\ve$ belongs to the essential spectrum of $\mathcal{L}_\ve$ for sufficiently small $\ve$
in the case $H(p_0)= m$, or $\lambda_\ve$ is the principal eigenvalue of $\mathcal{L}_\ve$ in the case $H(p_0)< m$, and
\begin{itemize}
\item[(i)] If  $H(p_0)= m$ then $\lambda_\ve=H(p_0)$ for sufficiently small $\ve$;  % and $\lambda_\ve$ belongs to the essential spectrum of $\mathcal{L}_\ve$.
\item[(ii)] If  $H(p_0)< m$ then %$\lambda_\ve$ is a principal eigenvalue of $\mathcal{L}_\ve$ for sufficiently small $\ve$, and
\begin{equation}
\label{PrEv}
\lambda_\ve= %\frac{1}{\ve^2}
H(p_0)+\Lambda_1\ve^2+o(\ve^2)\quad  as  \ \ve\to 0,
\end{equation}
where $\Lambda_1$ is the principal eigenvalue of the operator
\begin{equation}
\label{Hom_Dif_Op}
\mathcal{L}_0\rho=-{\rm div}\bigl(A\nabla \rho\bigr) \quad\text{in}\ \Omega, \quad \quad \rho=0\quad\text{on}\ \partial\Omega,
\end{equation}
whose matrix of coefficients $A$ has entries %=(A_{ij})$ is given by
\begin{equation}
A_{ij}=-\frac{1}{2}\partial^2_{p_i p_j} H(p_0).
\end{equation}
\end{itemize}
Moreover, in the case $H(p_0)< m$ the operator $\mathcal{L}_\ve$ has  a large or  infinite number of other eigenvalues $\lambda_\ve^{(j)}$ for small $\ve$.
Assuming that eigenvalues of  both $\mathcal{L}_\ve$
and $\mathcal{L}_0$ are arranged by their increasing real parts (and repeated according to their multiplicities), we have
\begin{equation}
\label{OtherEv}
\lambda_\ve^{(j)}= %\frac{1}{\ve^2}
H(p_0)+\Lambda_j \ve^2 +o(\ve^2)\quad  as  \ \ve\to 0,
\end{equation}
where $\Lambda_j$ are eigenvalues of the operator $\mathcal{L}_0$.%\eqref{Hom_Dif_Op}.
\end{thm}
%\begin{thm}
%\label{Golovna_periodic}
%Let $\lambda_\ve$ be the point of the spectrum of  $\mathcal{L}_\ve$ with the minimal real part, then,  as  $\ve\to 0$,
%\begin{align}
% &{\text if \ } H(p_0)< m,\quad  \lambda_\ve =\frac{1}{\ve^2} H(p_0)+O(1), \\
% &{\text if \ } H(p_0)= m,\quad	\lambda_\ve =\frac{1}{\ve^2} H(p_0).
%\end{align}
%Moreover, if $H(p_0)<m$ then for sufficiently small $\ve$ the following holds:
%\begin{itemize}
%\item $\lambda_\ve=\lambda_\ve^{(1)}$ is a simple eigenvalue of  $\mathcal{L}_\ve$ and
%it is given by the asymptotic formula
%$$
%\lambda_\ve= \frac{1}{\ve^2} H(p_0)+\Lambda_1+o(1)\quad  as  \ \ve\to 0,
%$$
%where $\Lambda_1$ is the first eigenvalue of the effective differential equation
%\begin{equation}
%A_{ij}\partial^2_{x_i x_j}\rho=\Lambda_1 \rho\quad\text{in}\ \Omega, \quad \rho=0\quad\text{on}\ \partial\Omega
%\end{equation}
%whose coefficients $A_{ij}$ are given by
%\begin{equation}
%A_{ij}=\frac{1}{2}\partial^2_{p_i p_j} H(p_0).
%\end{equation}
%\item for any $j=2,3\dots$ there are at least $j$ eigenvalues  of  $\mathcal{L}_\ve$, enumerated by the magnitude of their real parts, ${\text Re}(\lambda_\ve^{(j)}) \leq {\text Re}(\lambda_\ve^{(j)}) $ and repeated according their finite
%multiplicities
%\end{itemize}
% If $H(p_0)=m$ then $\lambda_\ve=\frac{1}{\ve^2} H(p_0)$ for sufficiently small $\ve$, and $\lambda_\ve$ belongs to the essential spectrum.
%\end{thm}
%\begin{rem}
%In \eqref{Hom_Dif_Op}
%\end{rem}
\begin{rem}
\label{Napryklad}
The following example inspired by \cite{Cov2010} shows that the case $H(p_0)=m$ and $H(p_0)<m$ do occur. Assume that $d\geq 3$, $\kappa=1$ and consider $J(z)$ given by
\begin{equation*}
J(z)=\frac{\mu e^{-|z|^2}}{\sum_{l\in \mathbb{Z}^d} e^{-|z+l|^2}},
\end{equation*}
where $\mu>0$. Then $\int_{\mathbb{R}^d} J(x-y)\varphi(y)dy=\mu\int_{\mathbb{T}^d} \varphi(y)dy$ for every  $\varphi\in C(\mathbb{T}^d)$.
Let $a(x)$ be a smooth periodic function strictly positive in $\mathbb{R}^d\setminus \mathbb{Z}^d$ and
such that $a(0)=0$ and $a(x)>\tau |x|^2$ ($\tau>0$) in a neighborhood of zero.
For such $J$, $\kappa$ and $a$,  we have $m=0$, while $H(0)=0$ if $\mu$ is small and $H(0)<0$ if $\mu$ is sufficiently large.
Indeed, according to  \cite{LiCovWan2017} it always holds that  $H(0)\leq 0$ and $H(0)<0$ iff $\exists$ $\varphi\in C(\mathbb{T}^d)$, $\varphi>0$ satisfying $-\mu\int_{\mathbb{T}^d} \varphi(y)dy+a(x)\varphi =H(0)\varphi$, i.e.
 $\varphi=\frac{1}{a(x)-H(0)}$ (up to multiplication by a positive constant) and $\int_{\mathbb{T}^d}\frac{\mu}{a(y)-H(0)} dy=1$. Thus $H(0)<0$ iff $\int_{\mathbb{T}^d}\frac{\mu}{a(x)} dx>1$.
\end{rem}
%\begin{rem} \textcolor{green} {Selfadjoint case: first periodic, then locally periodic.}
%In the  case when $J(z)=J(-z)$ the inequality $H(p_0)<m$  is satisfied if
%	 $M-m<\min\kappa(x,y) \,\int_{\mathbb{R}^d} J(z)dz$. It follows by using test function $\phi=1$ in \eqref{Hamilt_all_p_inf}. Another sufficient condition
%	 from \cite{LiCovWan2017}: \textcolor{red}{adapt (2.3)}.
%\end{rem}

\subsection{Problem reduction}%{Factorization, equivalent spectral problem.  \textcolor{red}{Modified spectral problem}}
%Providing factorization
%Similarly to

%\eqref{start_eq}

Similarly to spectral problems for differential operators with periodic oscillating coefficients \cite{Capdeb1998}, problem \eqref{start_eq}  can be transformed via a factorization trick to a form more convenient for the asymptotic analysis.
First we set
\begin{equation}
	\rho_\ve(x)=e^{-p\cdot x/\ve} u_\ve(x) \ \ \text{in} \ \Omega,
	\label{factor_u}
\end{equation}
%and extend $u_\ve$ by zero into $\mathbb{R}^d\setminus\Omega$. Then
so that the new unknown $u_\ve(x)$ satisfies
\begin{equation}
	-\frac{1}{\ve^{d}}\int_{\Omega} J\Bigl(\frac{x-y}{\ve}\Bigr)e^{\frac{1}{\ve}p\cdot (x-y)}\kappa\Bigl(\frac{x}{\ve},\frac{y}{\ve}\Bigr) u_\ve(y)dy+
	%\frac{1}{\ve^2}
	a\Bigl(\frac{x}{\ve}\Bigr) u_\ve (x)= \lambda_\ve  u_\ve (x) \quad \text{in} \ \Omega.
	\label{Persha_diya_baletu}
\end{equation}
Then consider a periodic counterpart of \eqref{Persha_diya_baletu} in the rescaled
variables $\xi=x/\ve$, $\eta=y/\ve$,
%and rescaling $x\to\frac{x}{\ve}$, $y\to\frac{y}{\ve}$ we rewrite \eqref{start_eq}  equation on torus
\begin{equation}
	-\int_{\mathbb{R}^d} J(\xi-\eta)e^{p\cdot (\xi-\eta)}\kappa(\xi,\eta) \varphi(\eta)d\eta+
	a(\xi) \varphi(\xi)= H(p) \varphi(\xi) \quad \text{in} \ \mathbb{T}^d.
	\label{direct_eq}
\end{equation}
Specifically, we are interested in the principal eigenvalue $H(p)$ (with the minimal real part), which, if exists, is real and simple,
its corresponding eigenfunction is sign preserving and thus can be chosen strictly positive, $\varphi>0$.
It is known \cite{LiCovWan2017} that  $H(p)$
given by \eqref{Hamilt_all_p_sup} %or \eqref{Hamilt_all_p_inf}
is always the principal eigenvalue of the problem \eqref{direct_eq}
provided that $H(p)<m$ (otherwise $H(p)=m$, it lies at the bottom point of the essential spectrum and  principal eigenvalue does not exist).  In the case $H(p)<m$ the adjoint problem
\begin{equation}
	-\int_{\mathbb{R}^d} J(\eta-\xi)e^{p\cdot (\eta-\xi)}\kappa(\eta,\xi) \varphi^*(\eta)d\eta+
	a(\xi) \varphi^*(\xi)= H(p) \varphi^*(\xi)\ \quad \text{in} \ \mathbb{T}^d
	\label{adjoint_eq}
\end{equation}
 has the same principal eigenvalue $H(p)$ and there also is a positive eigenfunction  $\varphi^*$.

Now we perform another change of the unknown
%It follows from \eqref{direct_eq} that one more change of variables
 %Set
%\begin{equation}
%m=\min a(x),\ M=\max a(x).
%\end{equation}
%We change $y=\frac{y}{\ve}$ in \eqref{direct_eq}, \eqref{adjoint_eq} and have
%\begin{equation}
%\frac{1}{\ve^{d}}\int_{\mathbb{R}^d} J\Bigl(x-\frac{y}{\ve}\Bigr)e^{p\cdot (y/\ve-x)}\kappa\Bigl(x,\frac{y}{\ve}\Bigr) \varphi\Bigl(\frac{y}{\ve}\Bigr)dy=\Bigl(a(x)-H(p) \Bigr) \varphi(x)\quad \text{in}\ \mathbb{T}^d,
%	\label{d_eq}
%\end{equation}
%\begin{equation}
%\frac{1}{\ve^{d}}\int_{\mathbb{R}^d} J\Bigl(\frac{y}{\ve}-x\Bigr)e^{p\cdot (x-y/\ve)}\kappa\Bigl(\frac{y}{\ve},x\Bigr) \varphi^*\Bigl(\frac{y}{\ve}\Bigr)dy=\Bigl(a(x)-H(p) \Bigr) \varphi^*(x)\quad \text{in}\ \mathbb{T}^d.
%	\label{a_eq}
%\end{equation}
\begin{equation}
	u_\ve(x)=\varphi\Bigl(\frac{x}{\ve}\Bigr) v_\ve(x),
	\label{factor_v}
\end{equation}
where $\varphi$ is a positive solution of  \eqref{direct_eq},  and introduce %the new  \textcolor{red}{shifted spectral parameter}
an affine change of the spectral parameter
\begin{equation}
\label{muuuuuu}
\mu_\ve=\frac{1}{\ve^2} \left(\lambda_\ve-H(p)\right)
\end{equation}
to transform
%problem
\eqref{start_eq}  to the spectral problem

%equation
%and subtracting $\frac{1}{\ve^{2}}H(p)v_\ve(x)$ we arrive to %spectral problem
%\begin{equation} -\frac{1}{\ve^{d+2}}\int_{\Omega} K\Bigl(\frac{x}{\ve},\frac{y}{\ve}\Bigr)v_\ve(y)dy+\frac{v_\ve(x)}{\ve^{d+2}}\int_{\mathbb{R}^d} K\Bigl(\frac{x}{\ve},\frac{y}{\ve}\Bigr)dy=\mu_\ve v_\ve(x)\quad \text{in}\ \Omega,\label{a_eq} \end{equation}
\begin{equation}
	\tilde{\mathcal{L}}_\ve v_\ve=\mu_\ve v_\ve \quad \text{in}\ \Omega,
	\label{or_a_eq}
\end{equation}
where
\begin{equation}
	\tilde{\mathcal{L}}_\ve v=-\frac{1}{\ve^{d+2}}\int_{\Omega} K\Bigl(\frac{x}{\ve},\frac{y}{\ve}\Bigr)
v(y)dy+\frac{1}{\ve^{d+2}}\int_{\mathbb{R}^d} K\Bigl(\frac{x}{\ve},\frac{y}{\ve}\Bigr) dy\, v(x),
 \label{krapka}
\end{equation}
and
\begin{equation}
K(x,y)=\frac{1}{\varphi(x)} J(x-y)e^{p\cdot (x-y)}\kappa(x,y) \varphi(y).
\label{ker_K}
\end{equation}
In what follows we will also deal with the kernel
\begin{equation}
	Q(x,y)=\varphi^*(x) \varphi(x) K(x,y),
	\label{ker_K'}
\end{equation}
by virtue of \eqref{direct_eq}--\eqref{adjoint_eq} this kernel  satisfies the following important property
\begin{equation}
\label{super_garne_property}
\int_{\mathbb{R}^d} Q(x,y)dy
	=\int_{\mathbb{R}^d} Q(y,x)dy\quad \forall x\in \mathbb{R}^d.
\end{equation}

Since the operators in problems \eqref{direct_eq} and \eqref{adjoint_eq} analytically depend on $p_i$, and $H(p)$ is a simple isolated eigenvalue if $H(p)<m$, then by perturbation theory \cite{Kato1966} $H(p)$ is an analytic function of $p_i$ and eigenfunctions $\varphi$, $\varphi^*$ can also be chosen analytic in $p_i$, $i=1,\dots, d$.

%Thus  if $H(p_0)<m$ we can take derivatives of   \eqref{adjoint_eq} at $p=p_0$ to establish the following

\begin{prop}
	Let $p_0$ be the maximum point of $H(p)$, and assume that $H(p_0)<m$. Then
the function $\chi_i^*:=\partial_{p_i}log\varphi^*\big|_{p=p_0}$ satisfies
\begin{equation}
	\int_{\mathbb{R}^d} Q\Bigl(\xi+z,\xi\Bigr)
	\Bigl(z_i+\chi_i^*(\xi+z)-\chi_i^*(\xi)\Bigr)dz=0 \quad \text{in} \ \mathbb{T}^d, \ i=1,\dots, d,
	\label{komirkova_zadacha}
\end{equation}
and
\begin{equation}
	\partial^2_{p_i p_j} H(p_0)=-\frac{1}{\int_{\mathbb{T}^d} \varphi(\xi)\varphi^*(\xi) d\xi}
	\int_{\mathbb{T}^d}	\int_{\mathbb{R}^d} Q\Bigl(\xi+z,\xi\Bigr)
	\Bigl(z_iz_j+2\chi_i^*(\xi+z)z_j\Bigr)dzd\xi.
	\label{do_komirkovoi}
\end{equation}
\end{prop}

\begin{proof}
To derive \eqref{komirkova_zadacha} we differentiate \eqref{adjoint_eq} with respect to $p_i$ at $p=p_0$, then multiply by $\varphi(\xi)$ and change the variables in the integral by setting $z=\eta-\xi$.
Similarly, \eqref{do_komirkovoi} is obtained by taking second derivatives of \eqref{adjoint_eq} and integrating the result over  $\mathbb{T}^d$ with the weight $\varphi(\xi)$.
\end{proof}

\begin{lem}
	Let $p_0$ be the maximum point of $H(p)$, and assume that $H(p_0)<m$. Then
\begin{equation}
		\partial^2_{p_i p_j} H(p_0)q_iq_j<0 \quad\forall q\in \mathbb{R}^d\setminus\{0\}.
		\label{lemma22}
	\end{equation}
\end{lem}
%\textcolor{red} {Einstain's summation condition}
\noindent Hereafter we assume summation over repeated indices.
\begin{proof}  By \eqref{do_komirkovoi} we have to show  positive difiniteness of the matrix  $A^\ast$
with entries
\begin{equation}
	 A_{ij}^*=
\int_{\mathbb{T}^d}\int_{\mathbb{R}^d} Q\Bigl(\xi+z,\xi\Bigr)
\Bigl(\frac{1}{2}z_iz_j+\chi_i^*(\xi+z)z_j\Bigr)dzd\xi,
\label{A_ij^star}
\end{equation}
where $\chi^*_i\in L^2(\mathbb{T}^d)$ are solutions of problems \eqref{komirkova_zadacha}.  To this
end for any $q\in \mathbb{R}^d$ we write, using  \eqref{komirkova_zadacha},
\begin{equation*}
\begin{aligned}
2A^\ast_{ij}q_iq_j=\int_{\mathbb{T}^d}\int_{\mathbb{R}^d} &Q(\xi+z,\xi)\left(
q_iz_i q_j z_j+2 q_i (\chi^*_i(\xi+z)-\chi^*_i(\xi)) q_jz_j \right)dz d\xi\\
&-2q_i q_j \int_{\mathbb{T}^d}\int_{\mathbb{R}^d}Q(\eta,\xi)\chi^*_i(\xi)(\chi^*_j(\eta)-\chi^*_j(\xi)) d\eta d\xi.
\end{aligned}
\end{equation*}
Thanks to \eqref{super_garne_property} we have
\begin{equation}
\begin{aligned}
-\int_{\mathbb{T}^d}\int_{\mathbb{R}^d}Q(\eta,\xi)\left(\chi^*_i(\xi)(\chi^*_j(\eta)-\chi^*_j(\xi)) +\chi^*_j(\xi)(\chi^*_i(\eta)-\chi^*_i(\xi)) \right)d\eta d\xi=\\
\int_{\mathbb{T}^d}\int_{\mathbb{R}^d}Q(\eta,\xi)\chi^*_i(\xi)\chi^*_j(\xi)d\eta d\xi +\int_{\mathbb{T}^d}\int_{\mathbb{R}^d}Q(\xi,\eta)\chi^*_i(\xi)\chi^*_j(\xi)d\eta d\xi \\
-\int_{\mathbb{T}^d}\int_{\mathbb{R}^d}Q(\eta,\xi)\left(\chi^*_i(\xi)\chi^*_j(\eta) +\chi^*_i(\eta)\chi^*_j(\xi)\right)d\eta d\xi,
\end{aligned}
\end{equation}
also
\begin{equation*}
\begin{aligned}
\int_{\mathbb{T}^d}\int_{\mathbb{R}^d}&Q(\xi,\eta)\chi^*_i(\xi)\chi^*_j(\xi)d\eta d\xi=\sum_{l\in\mathbb{Z}^d}\int_{\mathbb{T}^d\times\mathbb{T}^d}Q(\xi,\eta+l)\chi^*_i(\xi)\chi^*_j(\xi)d\eta d\xi
\\&=\sum_{l\in\mathbb{Z}^d}\int_{\mathbb{T}^d\times\mathbb{T}^d}Q(\xi-l,\eta)\chi^*_i(\xi-l)\chi^*_j(\xi-l)d\eta d\xi=
\int_{\mathbb{T}^d}\int_{\mathbb{R}^d}Q(\eta,\xi)\chi^*_i(\eta)\chi^*_j(\eta)d\eta d\xi.
\end{aligned}
\end{equation*}
Therefore	
$$
2A^\ast_{ij}q_iq_j=\int_{\mathbb{T}^d}\int_{\mathbb{R}^d} Q(\xi+z,\xi)q_i(z_i+\chi^*_i(\xi+z)-\chi^*_i(\xi))q_j(z_j+\chi^*_j(\xi+z)-\chi^*_j(\xi))d z d\xi\geq 0.
$$
The inequality is strict unless $q=0$, otherwise $q_i\chi^*_i(x)$ is a linear function whose gradient equals  $-q$ and hence  $q_i\chi^*_i(x)$ cannot be periodic if $q\not=0$.
\end{proof}

From now on we will assume that $p=p_0$, in particular,
this assumption will always be tacitly made when we refer to
\eqref{direct_eq}--\eqref{adjoint_eq}  and \eqref{muuuuuu}--\eqref{komirkova_zadacha}.
%$\phi$ and $\phi^*$ are referred to positive eigenfunctions of \eqref{direct_eq} and \eqref{adjoint_eq}
%with $p=p_0$
%$K(x,y)$ and $Q(x,y)$ are referred to the functions defined

\subsection{Resolvent convergence}
\label{sec_homohomo}
Assume that $H(p_0)<m$, and consider, for a given $f_\ve\in L^2(\Omega)$ the following problem
\begin{equation}
	\tilde{\mathcal{L}}_\ve v_\ve+ v_\ve=f_\ve\quad \text{in}\ \Omega,\quad
	 v_\ve=0\quad \text{in}\ \mathbb{R}^d\setminus\Omega,
	\label{inhomog_pr}
\end{equation}
where $\tilde{\mathcal{L}}_\ve$ is given by \eqref{krapka}.
Since $v_\ve=0$ in $ \mathbb{R}^d\setminus\Omega$ we can rewrite
\begin{equation}
	\tilde{\mathcal{L}}_\ve v_\ve=-\frac{1}{\ve^{d+2}}\int_{\mathbb{R}^d} K\Bigl(\frac{x}{\ve},\frac{y}{\ve}\Bigr)
	\Bigl(v_\ve(y)-v_\ve(x)\Bigr)dy.	
	\label{inhomog_pr_op}
\end{equation}
%(this operator appears in the spectral problem \eqref{or_a_eq}).

\begin{thm}
	\label{pro_resolventy}
There is a unique solution $ v_\ve(x)$ of the problem \eqref{inhomog_pr} in $L^2(\mathbb{R}^d)$ for any $f_\ve \in L^2(\Omega)$.
If $\|f_\ve\|_{L^2(\Omega)}\leq C$ with a constant $C$ independent of $\ve$ then the sequence of solutions $v_\ve$
contains a subsequence converging strongly in $L^2(\mathbb{R}^d)$ as $\ve\to 0$. If additionally $f_\ve\to f$
strongly in $L^2(\Omega)$ then the whole sequence of solutions $v_\ve$ converges
  to the unique solution of the problem
\begin{align}
	- A_{ij}\partial^2_{x_{i}x_{j}}v(x)+v(x)&=f(x)\quad \text{in}\ \Omega,
	\label{limit_pr_eq}
	\\
	v(x)&=0 \quad\quad \ \text{on}\ \partial\Omega,
	\label{limit_pr_bc}
\end{align}
extended by setting $v=0$ in $\mathbb{R}^d\setminus \Omega$.
\end{thm}
\begin{proof} It is convenient to extend $f_\ve(x)$ by zero into $\mathbb{R}^d\setminus\Omega$.
Observe that the Fredholm alternative applies to the problem \eqref{inhomog_pr} since the operator on the left hand side is represented as the sum of a compact operator and an invertible one.
	%operator multiplication by the function
	%$\frac{1}{\ve^{d+2}}\int_{\mathbb{R}^d} K\Bigl(\frac{x}{\ve},\frac{y}{\ve}\Bigr)dy+1$
To show that there is a solution of  \eqref{inhomog_pr} and to derive an a priori estimate multiply \eqref{inhomog_pr} by $\varphi(\frac{x}{\ve})\varphi^*(\frac{x}{\ve})v_\ve(x)$ and integrate over $\mathbb{R}^d$.
Using \eqref{super_garne_property} we obtain
\begin{equation}
\frac{1}{2\ve^{d+2}}\int_{\mathbb{R}^d}\int_{\mathbb{R}^d} Q\Bigl(\frac{x}{\ve},\frac{y}{\ve}\Bigr)
\left| v_\ve(y)-v_\ve(x)\right|^2  dydx
+\int_{\mathbb{R}^d} \left(\left| v_\ve(x) \right|^2- f_\ve(x)v_\ve(x) \right)\varphi\Bigl(\frac{x}{\ve}\Bigr)\varphi^*\Bigl(\frac{x}{\ve}\Bigr) dx=0
%=\int_{\mathbb{R}^d} f_\ve(x)v_\ve(x) \varphi(\frac{x}{\ve})\varphi^*(\frac{x}{\ve}) dx.
	 	%\label{or_a_eq}
\end{equation}
It follows that problem \eqref{inhomog_pr} cannot have nonzero solution for $f_\ve=0$. Thus \eqref{inhomog_pr} has a unique solution and using the Cauchy-Schwartz inequality we get
\begin{equation}
\int_{\mathbb{R}^d}\int_{\mathbb{R}^d} Q\Bigl(\frac{x}{\ve},\frac{y}{\ve}\Bigr)
		\left|v_\ve(y)-v_\ve(x)\right|^2 dydx
	\leq C\ve^{d+2}, \quad \|v_\ve\|_{L^2(\mathbb{R}^d)}\leq C,
		%\label{or_a_eq}
\end{equation}
with  a constant $C$ independant of $\ve$. Due to the fact that $Q(\xi,\eta)=\varphi^*(\xi)J(\xi-\eta)e^{p_0\cdot(\xi-\eta)}\varphi(\eta)$ and $J(0)>0$, we then %, in particular,
have
\begin{equation}
	\int_{\mathbb{R}^d}dx\int_{|z|\leq r_0\ve}|v_\ve(x+z)-v_\ve(x)|^2dz\leq C\ve^{d+2}
	\label{garnifunktsii}
\end{equation}
for some $r_0>0$ independent of $\ve$.
\begin{lem}
\label{Kruto_Lemma}
Let $v_\ve\in L^2(\mathbb{R}^d)$ be a sequence of functions satisfying \eqref{garnifunktsii} and such that $v_\ve =0$ in $\mathbb{R}^d\setminus {\Omega}$.
Then, up to extracting a subsequence, functions $v_\ve$ converge  strongly in $L^2(\mathbb{R}^d)$ to some limit $v$ as $\ve\to 0$. Moreover
$v\in H^1(\mathbb{R}^d)$ and $v =0$ in $\mathbb{R}^d\setminus{\Omega}$.
\end{lem}

\begin{proof} Without loss of generality  we can assume that $v_\ve\in C^\infty_0(\mathbb{R}^d)$. By Fubini's theorem
$$
\int_0^{r_0\ve}dr\int_{\mathbb{R}^d}dx\int_{|z|=r}|v_\ve(x+z)-v_\ve(x)|^2dS\leq C\ve^{d+2},
$$
therefore there exists $r_\ve$ such that $r_0\ve/2\leq r_\ve \leq  r_0 \ve$ and
\begin{equation}
\int_{\mathbb{R}^d}dx\int_{|z|=r_\ve}|v_\ve(x+z)-v_\ve(x)|^2dS\leq 2C\ve^{d+1}/r_0.
\label{tse_zh_treba}
\end{equation}
Consider functions
$$
\overline{v}_\ve(x)=\frac{1}{|B_1|\, r_\ve^d} \int_{|z|\leq r_\ve}v_\ve(x+z)dz.
$$
where $|B_1|=\frac{\Gamma( \frac{d}{2}+1)}{\pi^{{d}/{2}}}$ is the volume of the unit ball in $\mathbb{R}^d$.
 From \eqref{garnifunktsii} using Jensen's inequality we get
\begin{equation}
\int_{\mathbb{R}^d} | \overline{v}_\ve(x) -v_\ve(x)|^2dx \leq C \ve^2.
\label{NuPochalos'}
\end{equation}
Next observe that  $\forall$ $i=1,\dots, d$,
$$
\partial_{x_i} \overline{v}_\ve(x)= % \frac{\Gamma({\textstyle \frac{d}{2}+1})}{\pi^{{d}/{2}} r_\ve^d}
 \frac{1}{ |B_1| r_\ve^d}
\int_{|z|= r_\ve} v_\ve(x+z)\nu_i(z) dS
= %\frac{\Gamma({\textstyle \frac{d}{2}+1})}{2\pi^{{d}/{2}} r_\ve^d}
\frac{1}{ |B_1| r_\ve^d}  \int_{|z|= r_\ve} (v_\ve(x+z)-v_\ve(x-z))\nu_i(z) dS,
$$
where $\nu_i(z)={z_i}/{|z|}$ denotes the $i$-th component of the unite outward pointing normal to the $(d-1)$-sphere $|z|= r_\ve$.
Hence, using the Cauchy-Schwarz inequality we obtain
 $$
\left|\partial_{x_i} \overline{v}_\ve(x)\right|^2 \leq \frac{C}{\ve^{d+1}} \int_{|z|= r_\ve}  \left| v_\ve(x+z)-v_\ve(x)\right|^2 dS.
$$
Thus, thanks to \eqref{tse_zh_treba} we have
\begin{equation}
\int_{\mathbb{R}^d} \left|\nabla \overline{v}_\ve(x)\right|^2 dx \leq {C},
\label{NuVse}
\end{equation}
and since functions  $v_\ve$ vanish in $\mathbb{R}^d\setminus\Omega$ it holds that $\overline{v}_\ve=0$ in $\mathbb{R}^d\setminus\Omega^\prime$ for sufficiently small $\ve$,
where $\Omega^\prime$ is any bounded domain containing $\overline{\Omega}$. Then it follows from \eqref{NuVse} that, up to extracting a subsequence, functions $\overline{v}_\ve$ converge weakly
in  $H^1_{0}(\Omega^\prime)$ to a function $v\in H^1(\mathbb{R}^d)$ vanishing in $\mathbb{R}^d\setminus\overline{\Omega}$.
Thanks to the compactness of the embedding $H^1_{0}(\Omega^\prime)\subset L^2(\Omega^\prime)$  and  \eqref{NuPochalos'}
we also have the strong $L^2$-convergence of  functions  $v_\ve$ to $v$.
Finally, since $v=0$ in $\mathbb{R}^d\setminus\overline{\Omega}$ and $\partial\Omega$ is $C^1$-smooth, $v=0$ on $\partial \Omega$.
 Lemma is proved. \end{proof}
%{\it Proof of Theorem \ref{pro_resolventy} continued.}
We continue the proof of Theorem \ref{pro_resolventy}. By  Lemma \ref{Kruto_Lemma} we can  extract a subsequence of functions $v_\ve$
converging  strongly to some function $v\in H^1(\mathbb{R}^d)$ such that $v=0$ in  $\mathbb{R}^d\setminus\Omega$.
 Thus to complete the proof, it suffices to show that \eqref{limit_pr_eq} is satisfied in the sense of distributions.
 To this end consider an arbitrary $\phi\in C^\infty_0(\Omega)$ ($\phi=0$ in $\mathbb{R}^d\setminus\Omega$) and
 construct test functions $\phi_\ve(x)$
 %\textcolor{red}{with supports in $\Omega$ and}
 such that, as $\ve\to 0$,
\begin{align}
	-\frac{1}{\ve^{d+2}}\int_{\mathbb{R}^d} Q\Bigl(\frac{y}{\ve},\frac{x}{\ve}\Bigr)
	\left(\phi_\ve(y)-\phi_\ve(x)\right)dy
	&\rightharpoonup	- A_{ij}^*\partial^2_{x_{i}x_{j}}\phi(x) \ \   \text{weakly in}\  L^2(\Omega),
	\label{slaben'ko}
\\
\phi_\ve(x)	&\to	\phi(x) \ \   \text{strongly in}\  L^2(\Omega),
\label{syl'no}
\end{align}
where $A_{ij}^*$ are given by \eqref{A_ij^star}.  We set
\begin{equation}
\label{Oce_tak_ansatz}
\phi_\ve(x)=\phi(x)+\ve\partial_{x_i}\phi(x)\chi_i^*(x/\ve),
\end{equation}
where $\chi_i^*$ are solutions of  \eqref{komirkova_zadacha}. Then it is straightforward to see that \eqref{syl'no}
holds. To check \eqref{slaben'ko} perform changes of variables $x/\ve=\xi$, $y=x+\ve z$,
$$
-\frac{1}{\ve^{d+2}}\int_{\mathbb{R}^d} Q\Bigl(\frac{y}{\ve},\frac{x}{\ve}\Bigr)
	\left(\phi_\ve(y)-\phi_\ve(x)\right)dy= -\frac{1}{\ve^{2}}\int_{\mathbb{R}^d} Q\Bigl(\xi+z,\xi\Bigr)
	\left(\phi_\ve(x+\ve z)-\phi_\ve(x)\right)dz,
$$
and substitute the expansion
\begin{equation*}
\begin{aligned}
\phi_\ve(x+\ve z)=\phi_\ve(x)&+\ve(z_i+\chi_i^*(\xi+z)-\chi^*_i(\xi))\partial_{x_i}\phi(x)+\frac{\ve^2}{2}\partial^2_{x_ix_j}\phi(x)z_iz_j\\
&+\ve^2\chi_i^*(\xi+z)\left(z_j \partial^2_{x_ix_j}\phi(x) +O(\ve|z|^2)\right)+ O(\ve^3|z|^3).
%\bigl(\textstyle\frac{1}{2}z_iz_j+\chi_i^*(\xi+z)z_j\bigr)+\ve^3|z|^2 (O(|z|)+\chi_i^*(\xi+z) O(1))
\end{aligned}
\end{equation*}
Taking into account  \eqref{komirkova_zadacha} we find that
$$
-\frac{1}{\ve^{d+2}}\int_{\mathbb{R}^d} Q\Bigl(\frac{y}{\ve},\frac{x}{\ve}\Bigr)
	\left(\phi_\ve(y)-\phi_\ve(x)\right)dy= -\int_{\mathbb{R}^d} Q(\xi+z,\xi)
	\partial^2_{x_ix_j}\phi(x) \left(\textstyle\frac{1}{2}z_iz_j+\chi_i^*(\xi+z)z_j  \right)dz+O(\ve).
$$
Since functions $a^*_{ij}(\xi)=\int_{\mathbb{R}^d} Q(\xi+z,\xi)\left(\textstyle\frac{1}{2}z_iz_j+\chi_i^*(\xi+z)z_j  \right)dz$ are periodic, we have
$$
a^*_{ij}(x/\ve)\partial^2_{x_ix_j}\phi(x) \rightharpoonup \partial^2_{x_ix_j}\phi(x) \int_{\mathbb{T}^d} a^*_{ij}(\xi) d\xi \quad\text{weakly in} \ L^2(\Omega),
$$
 so that \eqref{slaben'ko} is also proved.

Now we can use $\varphi(x/\ve)\varphi^*(x/\ve)\phi_\ve(x)$ as a test function in  \eqref{inhomog_pr}  and pass to the limit as $\ve \to 0$. We have
\begin{equation*}
-\frac{1}{\ve^{d+2}}\int_{\mathbb{R}^d} v_\ve(x)
\int_{\mathbb{R}^d}
Q\Bigl(\frac{y}{\ve},\frac{x}{\ve}\Bigr)
	\Bigl(\phi_\ve(y)-\phi_\ve(x)\Bigr)dy=\int_{\mathbb{R}^d}\varphi\Bigl(\frac{x}{\ve}\Bigr) \varphi^*\Bigl(\frac{x}{\ve}\Bigr) (f_\ve(x)-	 v_\ve(x))\phi_\ve(x)dx,
\end{equation*}
whence we find in the limit $\ve\to 0$,
\begin{equation*}
-\int_{\mathbb{R}^d} v(x) A_{ij}^* \partial^2_{x_i x_j} \phi(x) dx =\int_{\mathbb{T}^d} \varphi(\xi) \varphi^*(\xi) d\xi \int_{\mathbb{R}^d} (f(x)-v(x))\phi(x)dx.
\end{equation*}
Thus, $v$ is a solution of the problem \eqref{limit_pr_eq}--\eqref{limit_pr_bc}, and thanks to its uniqueness the whole sequence of functions $v_\ve$ converge to $v$ strongly in $L^2(\mathbb{R}^d)$ as $\ve \to 0$. Theorem is proved.
\end{proof}

\subsection{Proof of Theorem \ref{Golovna_periodic}}

%It is known that
%\begin{equation}
%\end{equation}

Consider first the case $H(p_0)=m$. By \eqref{Hamilt_all_p_sup}  we have, $\forall \delta>0$ there is a function $\phi\in C(\mathbb{T}^d)$ such that
$$
-\int_{\mathbb{R}^d} J(\xi-\eta)e^{p_0\cdot \xi-\eta)}\kappa(\xi,\eta) \phi(\eta)d\eta+
		a(\xi) \phi(\xi)\leq (m-\delta) \phi(\xi).
$$
On the other hand
\begin{equation*}
\begin{aligned}
\lambda_\ve=\sup\Biggl\{\lambda \,\Bigl.\Bigr| \Biggr.&\, \exists \rho\in C(\overline{\Omega}),\, \rho>0, \ \text{such that} \\
&\Biggl.
-\frac{1}{\ve^{d}}\int_{\Omega} J\Bigl(\frac{x-y}{\ve}\Bigr)%e^{p\cdot (y-x)}
\kappa\Bigl(\frac{x}{\ve},\frac{y}{\ve}\Bigr) \rho(y)dy+
%\frac{1}{\ve^2}
a\Bigl(\frac{x}{\ve}\Bigr) \rho(x)\geq \lambda \rho(x) \ \text{in}\ \overline{\Omega} \Biggr\},
 \end{aligned}
%\label{Hamilt_all_p_sup}
\end{equation*}
then taking $\rho(x)=e^{p_0\cdot x/\ve}\phi(x/\ve)$ we see that $\lambda_\ve\geq m-\delta$. Thus $\lambda_\ve\geq m$ and
if $\min_{x\in\overline{\Omega}} a(x/\ve)=m$ (that is always true for sufficiently small $\ve$) then $\lambda_\ve = m$ and it belongs to the
essential spectrum of $\mathcal{L}_\ve$.

Now consider the case $H(p_0)<m$. Let $v_\ve$ be an eigenfunction corresponding to an arbitrary (not necessarily principal) eigenvalue $\mu_\ve$.
First we show that if the real part of $\mu_\ve$ is bounded then $|\mu_\ve|$ is bounded. To this end multiply \eqref{or_a_eq} by $\varphi(\frac{x}{\ve})\varphi^*(\frac{x}{\ve})\overline v_\ve(x)$,
where  $\overline v_\ve$  denotes the complex conjugate,
take real part and integrate over $\mathbb{R}^d$. Then using  \eqref{super_garne_property} we obtain the following equality
\begin{equation}
\label{otsinka_scho_nado}
\frac{1}{2\ve^{d+2}}\int_{\mathbb{R}^d}\int_{\mathbb{R}^d} Q\Bigl(\frac{x}{\ve},\frac{y}{\ve}\Bigr)
\left| v_\ve(y)-v_\ve(x)\right|^2 dydx
%+\int_{\mathbb{R}^d} \left| v_\ve(x) \right|^2 dx
={\rm Re}\mu_\ve\int_{\mathbb{R}^d} \left| v_\ve(x) \right|^2 \varphi\Bigl(\frac{x}{\ve}\Bigr)\varphi^*\Bigl(\frac{x}{\ve}\Bigr)dx
	 	%\label{or_a_eq}
\end{equation}
Therefore if $\|v_\ve\|_{L^2(\Omega)}=1$ then  \eqref{garnifunktsii}
holds and by Lemma \ref{Kruto_Lemma} functions converge strongly in $L^2(\Omega)$ to a (nonzero) function $v$ as  $\ve\to 0$ along a subsequence. Then introducing $\tilde v_\ve=\frac{1}{1+\mu_\ve}v_\ve$ and passing to the limit in the equality $\tilde{\mathcal{L}}_\ve \tilde v_\ve+ \tilde v_\ve=v_\ve$ we obtain by virtue of Theorem
\ref{pro_resolventy} that functions $\tilde v_\ve$ converge to a solution $\tilde v$ of the problem
${\mathcal{L}}_0 \tilde v=v$. One the other hand $\tilde v_\ve\to 0$ (since $|\mu_\ve|\to\infty$), thus $v=0$, a contradiction.

Next we show that for $\mu$ from every compact subset $M$ of $\mathbb{C}\setminus\cup_{k=1}^\infty \{\Lambda_k\}$ and sufficiently small $\ve$  the operator $\bigl(\mu I-\tilde{\mathcal{L}}_\ve\bigr)^{-1}: L^2(\Omega)\to L^2(\Omega)$ exists and its operator norm is uniformly bounded. Indeed, otherwise there exist functions $v_\ve$ with $\|v_\ve\|_{L^2(\Omega)}=1$ and numbers $\mu_\ve\in M$ such that
%$g_\ve=$
$\tilde{\mathcal{L}}_\ve v_\ve-\mu_\ve v_\ve\to 0$ strongly in $L^2(\Omega)$ as $\ve\to 0$ (along a subsequence). Then, by  Theorem~\ref{pro_resolventy}
one can extract a subsequence of functions $v_\ve$ converging to a nontrivial solution $v$ of the equation ${\mathcal{L}}_0 v-\mu v=0$ for some $\mu\in M$, that is impossible
since $M\cap \cup_{k=1}^\infty\{\Lambda_k\}=\emptyset$. Moreover, applying Theorem  \ref{pro_resolventy} we conclude that $\forall f\in L^2(\Omega)$ and $\mu\in  M$, $\bigl(\mu I-\tilde{\mathcal{L}}_\ve\bigr)^{-1}f\to \bigl(\mu I-{\mathcal{L}}_0\bigr)^{-1}f$ strongly
in $L^2(\Omega)$ as $\ve\to 0$, therefore spectral projectors $\Pi_\ve(\omega)=\frac{1}{2\pi i}\int_{\partial \omega}\bigl(\mu I-\tilde{\mathcal{L}}_\ve\bigr)^{-1} d\mu$ converge strongly to  the projector
$\Pi_0(\omega)=\frac{1}{2\pi i}\int_{\partial \omega}\bigl(\mu I-\mathcal{L}_0\bigr)^{-1} d\mu$ for any bounded open set $\omega\subset\mathbb{C}$ whose boundary is smooth and does not contain eigenvalues $\Lambda_k$.
In fact, there is the compact convergence of projectors, i.e. additionally to the strong convergence it holds that
for any sequence of functions $f_\ve$ bounded in $L^2(\Omega)$ the sequence of projections  $v_\ve = \Pi_\ve(\omega) f_\ve$ contains
a strongly converging subsequence.  Indeed, observe that  $v_\ve$ satisfy
\begin{equation*}
%\label{dlinyuka}
\tilde{\mathcal{L}}_\ve v_\ve+v_\ve=\frac{1}{2\pi i}\int_{\partial \omega}(1+\mu)\bigl(\mu I-\tilde{\mathcal{L}}_\ve\bigr)^{-1} f_\ve d\mu
\end{equation*}
and thanks to  the uniform boundedness of $\bigl(\mu I-\tilde{\mathcal{L}}_\ve\bigr)^{-1}$  Theorem  \ref{pro_resolventy} guarantees that the sequence
of functions  $v_\ve$ does contain a strongly converging (as $\ve\to 0$) subsequence. The compact convergence of spectral projectors in turn implies that the
dimensions of the subspaces $\Pi_\ve(\omega) L^2(\Omega)$ and  $\Pi_0(\omega) L^2(\Omega)$ coincide as $\ve$ is sufficiently small, i.e.
operators $\tilde{\mathcal{L}}_\ve$ and ${\mathcal{L}}_0$ have the same number of eigenvalues (counting multiplicities) in the domain $\omega$.
This means, in particular, that there is an eigenvalue of $\tilde{\mathcal{L}}_\ve$ converging to $\Lambda_1$ as $\ve\to 0$. Therefore
the principal eigenvalue of $\tilde{\mathcal{L}}_\ve$ exists for sufficiently small $\ve$. It remains bounded  as $\ve\to 0$ since its real part remains bounded,
and it converges (up to a subsequence) to an eigenvalue of ${\mathcal{L}}_0$ (any compact subset of  $\mathbb{C}\setminus\cup_{k=1}^\infty \{\Lambda_k\}$ belongs to the resolvent
set of $\tilde{\mathcal{L}}_\ve$ for sufficiently small $\ve$). Thus the principal eigenvalue of $\tilde{\mathcal{L}}_\ve$ converges to $\Lambda_1$. Other eigenvalues
can be treated similarly. Theorem~\ref{Golovna_periodic} is proved.

\section{Locally periodic case}
\label{Locallyperiodiccase}

This section is devoted to the proof of Theorem \ref{Golovna_locperiodic}, i.e. we study operator $\mathcal{L}_\ve$ given by
\eqref{main_operator_locper}  with generic functions $\kappa$ and $a$ satisfying \eqref{akappa}.  Introduce the notation
\begin{equation*}
E_\ve(x,y)=\frac{1}{\ve^d}J\Bigl(\frac{x-y}{\ve}\Bigr)\kappa\Bigl(x,y,\frac{x}{\ve}, \frac{y}{\ve}\Bigr),
% -\frac{1}{\ve^d} \int_{\Omega} J\Bigl(\frac{x-y}{\ve}\Bigr)\kappa\Bigl(x,y,\frac{x}{\ve}, \frac{y}{\ve}\Bigr) \rho_\ve(y) dy +a\Bigl(x,\frac{x}{\ve}\Bigr)\rho_\ve.
\end{equation*}
then $\mathcal{L}_\ve$ writes as
\begin{equation}
 \label{loc_per_op}
\mathcal{L}_\ve\rho_\ve= -\int_{\Omega} E_\ve(x,y) \rho_\ve(y) dy +a\Bigl(x,\frac{x}{\ve}\Bigr)\rho_\ve(x),
%\quad \text{where}\
% E_\ve(x,y)=\frac{1}{\ve^d}J\Bigl(\frac{x-y}{\ve}\Bigr)\kappa\Bigl(x,y,\frac{x}{\ve}, \frac{y}{\ve}\Bigr),
% -\frac{1}{\ve^d} \int_{\Omega} J\Bigl(\frac{x-y}{\ve}\Bigr)\kappa\Bigl(x,y,\frac{x}{\ve}, \frac{y}{\ve}\Bigr) \rho_\ve(y) dy +a\Bigl(x,\frac{x}{\ve}\Bigr)\rho_\ve.
\end{equation}
and
%
%written for brevity as
%%
%%we con
%%In the locally periodic case  we consider continuous functions $\kappa=\kappa(x,y,\xi,\eta)$ and $a=a(x,\xi)$ ($\xi=x/\ve$, $\eta=y/\ve$) and assume that
%%they are periodic in $\xi$ and $\eta$. The nonlocal operator  $\mathcal{L}_\ve$ now takes form
% \begin{equation}
% \label{loc_per_op}
%\mathcal{L}_\ve\rho_\ve= -\int_{\Omega} E_\ve(x,y) \rho_\ve(y) dy +a\Bigl(x,\frac{x}{\ve}\Bigr)\rho_\ve(x),
%\quad \text{where}\
% E_\ve(x,y)=\frac{1}{\ve^d}J\Bigl(\frac{x-y}{\ve}\Bigr)\kappa\Bigl(x,y,\frac{x}{\ve}, \frac{y}{\ve}\Bigr),
%% -\frac{1}{\ve^d} \int_{\Omega} J\Bigl(\frac{x-y}{\ve}\Bigr)\kappa\Bigl(x,y,\frac{x}{\ve}, \frac{y}{\ve}\Bigr) \rho_\ve(y) dy +a\Bigl(x,\frac{x}{\ve}\Bigr)\rho_\ve.
%\end{equation}
%with generic functions $\kappa$ and $a$ satisfying \eqref{akappa}.
 %We consider the spectrum of the operator $\mathcal{L}_\ve$ and study the asymptotic behavior as $\ve\to 0$ of
%bottom point $\lambda_\ve$ of the spectrum,
\begin{equation}
%\begin{aligned}
\lambda_\ve=\sup\Biggl\{\lambda \,\bigl.\Bigr| \Biggr.%&
\, \exists v\in C(\overline\Omega),\, v>0 \ \text{such that} \
%&
\Biggl.
%-\frac{1}{\ve^d}
-\int_{\Omega}  E_\ve(x,y)%J\Bigl(\frac{x-y}{\ve}\Bigr)\kappa\Bigl(x,y,\frac{x}{\ve}, \frac{y}{\ve}\Bigr)
\frac{v(y)}{v(x)}dy+
 a\Bigl(x,\frac{x}{\ve}\Bigr) \geq \lambda  \ \text{in}\ \Omega \Biggr\}.
% \end{aligned}
\label{Pr_e_v_supformula_bis}
\end{equation}
Assume first that $\lambda_\ve<\min_{\overline{\Omega}} a(x,x/\ve)$, then
$\lambda_\ve$ is the principal eigenvalue of  $\mathcal{L}_\ve$.  Therefore the corresponding eigenfunction can be
written as $\rho_\ve=e^{-\frac{1}{\ve} W_\ve(x)}$. We consider the following ansatz for $W_\ve$, $W_\ve(x)= W(x)+\ve w(x,x/\ve)+\dots$,
where $w(x,\xi)$ is periodic in $\xi$. Together with the fast variable $\xi=x/\ve$  we also introduce $\eta=y/\ve$ and regard these variables as
independent of the slow ones, $x$ and $y$. We also hypothesize that $\lambda_\ve$ converges to a finite number $-\Lambda$.
%Next, fix $x\in \Omega$ and introduce
%the fast variables $\xi=x/\ve$, $\eta=y/\ve$.
Then, for fixed  $x\in \Omega$ we expand
$$
W(y)=W(x+\ve(\eta-\xi))= W(x)+\ve\nabla W(x)\cdot (\eta-\xi)+\dots
$$
and  formally obtain in the leading term of the eigenvalue equation $\mathcal{L}_\ve\rho_\ve=\lambda_\ve \rho_\ve$ that
$\Lambda=-H(\nabla W(x),x)$,
%to the leading order $\lambda_\ve = -\Lambda$, $-H(\nabla W(x),x)$ and $w(x,\xi)=\varphi(\xi,\nabla W(x),x)$ with
  $H(p,x)$ being the principal eigenvalue of the cell problem
 %and $\varphi(\xi,p,x)>0$ being a
%and considering $x$, $y$, $\xi$ and $\eta$ as independent variables (as usual in formal asymptotic expansions)
%eigenfunction $\varphi(\xi,p,x)$ of the equation
\begin{equation}
\label{typu_cell_pr-m}
-\int_{\mathbb{R}^d} J(\xi- \eta)e^{p\cdot (\xi-\eta)}\kappa(x,x,\xi,\eta) \varphi(\eta,p,x)d \eta+
 a(x, \xi) \varphi(\xi,p,x)=H(p,x) \varphi(\xi,p,x) \ \text{on}\ \mathbb{T}^d
 \end{equation}
depending on the parameters $p\in\mathbb{R}^d$ and $x\in\overline{\Omega}$, while $w(x,\xi)=-\log\varphi(\xi,\nabla W(x),x)$.
Adopting the normalization condition $\int_{\mathbb{T}^n}\varphi(\xi,p,x) d\xi =1$ we obtain a function $\varphi(\xi,p,x)$ continuous in all their arguments ($\xi$, $p$ and $x$), provided that the principal
eigenvalue exists. Notice that the principal eigenvalue $H(p,x)$ is given by  \eqref{Hamilt_all_p_sup_dep_on_x} and always satisfies
 $H(p,x)<\min_{\xi\in\mathbb{T}^d} a(x,\xi)$, moreover the latter inequality is  sufficient and
 necessary for existence of a principal eigenvalue of  \eqref{typu_cell_pr-m}.  We show below
 that  if
 \begin{equation}
 \label{zaiva_umova}
 \lambda_\ve<\min_{\overline{\Omega}} a(x,x/\ve)\ \text{and} \  H(p,x)<\min_{\xi\in\mathbb{T}^d} a(x,\xi)
 \end{equation}
 then $\lambda_\ve\to -\Lambda$ as $\ve\to 0$,
 where $\Lambda$ is in fact the minimal eigenvalue of the problem $-H(\nabla W(x),x)=\Lambda$ in $\Omega$ or, equivalently,
 \begin{equation*}
 \Lambda=\min \big\{\tilde\Lambda\,|\, \exists\ \text{a viscosity subsolution of} \ -H(\nabla W(x),x)\leq \tilde\Lambda \ \text{in}\ \Omega\bigr\}.
 \end{equation*}
 As known, see, e.g.,  \cite{Mit2008} this formula  (along with  \eqref{Lambda_pershaformulka} and \eqref{Diya})  determines the unique additive eigenvalue $\Lambda$ of problem \eqref{HJ_lim_eq}.

The additional technical assumptions \eqref{zaiva_umova} will then be eliminated by devising small deformations of
$a(x,\xi)$ regularizing eigenvalue problems, and in this way we will get the proof of Theorem \ref{Golovna_locperiodic}.

%
%Moreover,
%introducing
%%\begin{equation}
%%H(p,x)=-\int_{\Omega} J(z)e^{p\cdot z}\kappa(x,x)  dz +a(x)
%%\end{equation}
%% (in the locally periodic case
% \begin{equation}
%\begin{aligned}
%H(p,x)=\sup\Biggl\{\lambda \,\Bigl.\Bigr| \Biggr.&\, \exists \varphi\in C(\mathbb{T}^d),%_{\#}
%\, \varphi>0, \ \text{such that} \\
%&\Biggl.
%-\int_{\mathbb{R}^d} J(z)e^{p\cdot z}\kappa(x,x,\xi,\xi-z) \varphi(\xi-z)d z+
% a(x, \xi) \varphi(\xi)\geq \lambda \varphi(\xi) \ \text{on}\ \mathbb{T}^d \Biggr\},
% \end{aligned}
%\label{Hamilt_all_p_sup_dep_on_x}
%\end{equation}
%we also assume that $H(p,x)<\min_{\xi\in\mathbb{T}^d} a(x,\xi)$. This guarantees  existence of a positive
%continuous periodic eigenfunction $\varphi(\xi,p,x)$ of the equation
%\begin{equation}
%\label{typu_cell_pr-m}
%-\int_{\mathbb{R}^d} J(\xi- \eta)e^{p\cdot (\xi-\eta)}\kappa(x,x,\xi,\eta) \varphi(\eta,p,x)d \eta+
% a(x, \xi) \varphi(\xi,p,x)=H(p,x) \varphi(\xi,p,x) \ \text{on}\ \mathbb{T}^d,
% \end{equation}
%and normalizing $\varphi(\xi,p,x)$ by $\int_{\mathbb{T}^n}\varphi(\xi,p,x) d\xi =1$ we select a family of
%eigenfunctions continuous in $p$ and $x$ as well. In the sequel we will get rid of the assumptions made above
%about existence of principal eigenvalues by introducing regularized

\begin{thm}
\label{localno_periodychna_teorema}
 %{main_operator_locper}
Suppose that $J$ satisfies
\eqref{umova_dva} and $\kappa$, $a$ satisfy \eqref{akappa}.  Assume also that
$\lambda_\ve$ (given by \eqref{Pr_e_v_supformula}) and $H(p,x)$ (given by  \eqref{Hamilt_all_p_sup_dep_on_x}) satisfy  \eqref{zaiva_umova}.
% $\kappa\in C(\Omega\times\Omega \times \mathbb{T}^d\times \mathbb{T}^d)$,  $\kappa>0$,
%$a\in C(\Omega\times \mathbb{T}^d)$ and
Then $\lambda_\ve\to -\Lambda$ as $\ve\to 0$,
where $\Lambda$ is a unique additive eigenvalue of problem \eqref{HJ_lim_eq}.
 \end{thm}

%Assume that  $\lambda_{\ve}< m$ (therefor $\lambda_\ve$ is the principal eigenvalue
 %of  \eqref{}), then
% \begin{equation}
 %\limsup_{\ve\to 0}\lambda_{\ve}< m=\min a(x).
 %\label{Pereprypuschennya}
 %\end{equation}
 %Then $\lambda_\ve$ is the principal eigenvalue  for sufficiently small $\ve$, and

\begin{proof} We begin with the following lower bound, obtained by using the test function $$v_\ve(x)=e^{-\frac{1}{\ve}W(x)}\varphi(x/\ve,\nabla W(x),x)$$ in \eqref{Pr_e_v_supformula_bis},
$$
\lambda_\ve\geq \min_{x\in\overline{\Omega}}\Bigl\{-\int_{\Omega}  E_\ve(x,y)%J\Bigl(\frac{x-y}{\ve}\Bigr)\kappa\Bigl(x,y,\frac{x}{\ve}, \frac{y}{\ve}\Bigr)
\frac{v_\ve(y)}{v_\ve(x)}dy+
 a\Bigl(x,\frac{x}{\ve}\Bigr)\Bigr\},
$$
where  $W$ is an arbitrary function of the class $C_0^\infty(\mathbb{R}^d)$. Notice that uniformly in $x\in\Omega$,
\begin{equation*}
\begin{aligned}
-\int_\Omega%{\mathbb{R}^d}
E_\ve(x,y)%J\Bigl(\frac{x-y}{\ve}\Bigr)\kappa\Bigl(x,y,\frac{x}{\ve}, \frac{y}{\ve}\Bigr)
\frac{v_\ve(y)}{v_\ve(x)}dy&+
 a\Bigl(x,\frac{x}{\ve}\Bigr)=  a\Bigl(x,\frac{x}{\ve}\Bigr)+o(1) \\
& -\int_{\frac{x}{\ve}-\frac{1}{\ve}\Omega}
 %{\mathbb{R}^d}
 J(z)\kappa\Bigl(x,x,\frac{x}{\ve}, \frac{x}{\ve}-z\Bigr)e^{\frac{1}{\ve}(W(x)-W(x-\ve z))}
 %J\Bigl(\frac{x-y}{\ve}\Bigr)\kappa\Bigl(x,y,\frac{x}{\ve}, \frac{y}{\ve}\Bigr)
\frac{\varphi(x/\ve-z,\nabla W(x),x)}{\varphi(x/\ve,\nabla W(x),x)}dz.
\end{aligned}
\end{equation*}
Expanding $W(x-\ve z)=W(x)-\ve \nabla W(x)\cdot z +O(\ve^2 |z|^2)$, using  \eqref{typu_cell_pr-m} and taking into account   \eqref{umova_dva}
%\textcolor{red}{(need to elaborate this a little)}
we get
\begin{equation}
\label{first_good_estimate}
\liminf_{\ve\to 0} \lambda_\ve \geq \min_{x\in \overline\Omega} H(\nabla W(x),x).
\end{equation}
Therefore by density of functions $W|_{\overline{\Omega}}$, $W\in C^\infty_0(\mathbb{R}^d)$ in $C^1(\overline{\Omega})$ we have
\begin{equation}
\label{first_good_otsinka_final}
\liminf_{\ve\to 0} \lambda_\ve \geq -\Lambda, \quad \text{where}\ \Lambda=\inf_{W\in C^1(\overline{\Omega})}\max_{x\in \overline\Omega} -H(\nabla W(x),x).
\end{equation}

Next, considering a partial limit $\lambda$ of $\lambda_\ve$ as $\ve\to 0$, we use the techniques of half-relaxed limits (introduced  in \cite{BarPer1987}) to show that there is a viscosity
subsolution $W^\ast (x)$ of
 \begin{equation}
 \label{subsol_partial_lambda}
 -H(\nabla W^\ast (x), x)\leq -\lambda\quad\text{in}\ \Omega.
 \end{equation}
Specifically, let $e^{-\frac{1}{\ve} W_\ve(x)}$ be the eigenfunction of $\mathcal{L}_\ve$ corresponding to the eigenvalue $\lambda_\ve$
and assume that this function satisfies the following normalization condition
\begin{equation}
 \label{seredne_nul}
 \int_{\Omega^\prime} W_\ve(x)dx=0,
 \end{equation}
where $\Omega^\prime$ is a domain such that $\overline{\Omega^\prime}\subset\Omega$.  Since $\lambda_\ve<\min_{x\in \overline {\Omega}} a(x,x/\ve)$ and \eqref{first_good_otsinka_final} holds,
we can assume, after passing to a subsequence that  $\lambda_\ve\to \lambda$ as $\ve\to 0$. Then we consider the half-relaxed limit  %the half-relaxed limits method being introduced in.
 \begin{equation}
 W^*(x)=\lim_{r\to 0}\limsup_{\ve \to 0} \sup\{W_\ve(\xi) | \, \xi \in B_r(x)\cap\Omega\}.
 \label{relaxed}
\end{equation}
 \begin{lem}
 \label{HighTechLemma}
 Assume that functions $W_\ve(x)$ satisfy \eqref{seredne_nul}.
 Then $W^*(x)$ given by \eqref{relaxed} is a bounded %(upper semicontinuous)
 function in $\Omega$.
 \end{lem}

 \begin{proof}  As $J(0)>0$ and $J$ is continuous,  %Thanks to \eqref{umova_dva}
 there is $\hat r_0>0$ such that $\inf_{z\in B_{\hat r_0}}J(z)>0$. Furthermore, since $\partial\Omega$ is $C^1$-smooth, there
 is $\tilde r_\ve\geq c\ve$ (with $c>0$ independent of $\ve$) such that any ball $B_{\hat r_0 \ve}(x)$ centered at a point $x\in\overline{\Omega}$
 contains a ball  $B_{\tilde r_\ve}(\xi)$ that is also contained in $\Omega$, i.e. $B_{\tilde r_\ve}(\xi)\subset B_{\hat r_0 \ve}(x)\cap \Omega$.

 We argue as in Lemma \ref{Kruto_Lemma}. Thanks  to \eqref{first_good_otsinka_final} eigenvalues $\lambda_\ve$ are uniformly  bounded from below and  we have
 $$
\int_\Omega\int_{\Omega}  J\Bigl(\frac{x-y}{\ve}\Bigr)\kappa\Bigl(x,y,\frac{x}{\ve},\frac{y}{\ve}\Bigr)e^{\frac{1}{\ve}(W_\ve(x)-W_\ve(y))}dxdy\leq C\ve^d,
 $$
 therefore
$$
\int_{\Omega_\ve} \int_{|z|<\tilde r_\ve}  e^{\frac{1}{\ve}|W_\ve(x+z)-W_\ve(x)|}dzdx\leq C\ve^d,
 $$
where $\Omega_\ve=\{x\in\Omega\,|\,{\rm dist}(x,\partial\Omega)>\tilde r_\ve\}$. In particular,
 \begin{equation}
 \label{persha_a_priori}
\int_{\Omega_\ve} \int_{|z|<\tilde r_\ve}  |W_\ve(x+z)-W_\ve(x)|^{d+1}dzdx\leq C\ve^{2d+1}.
\end{equation}
It follows that there is some $r_\ve$, $\tilde r_\ve/2\leq r_\ve \leq \tilde r_\ve$ such that
 $$
\int_{\Omega_\ve} \int_{|z|= r_\ve}  |W_\ve(x+z)-W_\ve(x)|^{d+1}dSdx\leq C\ve^{2d}.
 $$
Set
 $$
\overline{W}_\ve(x)=\frac{1}{|B_1|r_\ve^d}\int_{|z|< r_\ve}  W_\ve(x+z)dz.
 $$
 Using Jensen's inequality we get
 \begin{equation}
 \label{Dopomoga}
 \int_{\Omega_\ve}
 |\overline{W}_\ve(x)-W_\ve(x)|^{d+1}dx
 \leq C \ve^{d+1}.
\end{equation}
 Then, arguing as in Lemma \ref{Kruto_Lemma} we derive
  \begin{equation}
  \label{DuzheGarnoGradW}
 \int_{\Omega_\ve}
 |\nabla \overline{W}_\ve(x)|^{d+1}dx
 \leq C.
\end{equation}
Now, taking into account \eqref{seredne_nul}, \eqref{Dopomoga} we can apply the Poincar\'e inequality to
conclude that %norms of functions $\overline{W}_\ve(x)$  in $W^{1,d+1}(\Omega_\ve)$ are uniformly bounded.
$\int_{\Omega_\ve}
 |\overline{W}_\ve(x)|^{d+1}dx
 \leq C$ for small $\ve$.
%, and using \eqref{Dopomoga} one more time we obtain
% \begin{equation}
%  \label{DuzheGarno_takozhW}
% \int_{\Omega_\ve}
% |{W}_\ve(x)|^{d+1}dx
% \leq C.
%\end{equation}
Then by the compactness of the embedding $W^{1,d+1}(\Omega_\ve)\subset C(\overline{\Omega}_\ve)$
(Morrey's theorem) we derive that $|\overline{W}_\ve(x)|\leq C$ on $\Omega_\ve$ with $C$ independent
of $\ve$.  Combining this with \eqref{Dopomoga}  we infer that $W^\ast(x)$ is bounded from below.

Repeating the above reasonings for the positive part  $W^+_\ve(x)$ of $W_\ve(x)$ (notice that the inequality  \eqref{persha_a_priori} is also valid for
$W^+_\ve(x)$) we get that
$\overline{W}_\ve^+(x)=\frac{1}{|B_1|r_\ve^d}\int_{|z|< r_\ve}  W_\ve^+(x+z)dz$ satisfies %(cf.  \eqref{DuzheGarnoGradW}, \eqref{DuzheGarno_takozhW})
%\begin{equation*}
%  \label{UmnitsaMoya}
$\int_{\Omega_\ve} |\nabla \overline{W}_\ve^+(x)|^{d+1} dx \leq C$. Besides, using \eqref{Dopomoga}  one sees that $\int_{\Omega^\prime} \overline{W}_\ve^+(x) dx \leq C$.
Then applying the Poincar\'e inequality and exploiting the compactness of the embedding $W^{1,d+1}(\Omega_\ve)\subset C(\overline{\Omega}_\ve)$   we obtain that
$\overline{W}_\ve^{+}(x) \leq C$ on $\Omega_\ve$ with $C$ independent
of $\ve$.
%\right) dx
%\left( |\nabla \overline{W}_\ve^+(x)|^{d+1} +|\overline{W}_\ve^{+}(x)|^{d+1}
%\right) dx
%\leq C.
%\end{equation*}

Taking $\log$ of
 $$
 \frac{1}{\ve^d}\int_{\Omega}  J\Bigl(\frac{x-y}{\ve}\Bigr)\kappa\Bigl(x,y,\frac{x}{\ve},\frac{y}{\ve}\Bigr) e^{-\frac{1}{\ve}W_\ve(y)}dy=
(a(x)-\lambda_\ve)  e^{-\frac{1}{\ve}W_\ve(x)}
 $$
 and using Jensen's inequality we get
 $$
 W_\ve(x)\leq  \frac{1}{\int_{\Omega}  J\left(\frac{x-y}{\ve}\right)\kappa(x,y,\frac{x}{\ve},\frac{y}{\ve})dy}\int_{\Omega}  J\Bigl(\frac{x-y}{\ve}\Bigr)\kappa \Bigl(x,y,\frac{x}{\ve},\frac{y}{\ve}\Bigr)  W_\ve(y)dy+\ve R_\ve(x),
$$
where
 $$
 R_\ve(x)= \log(a(x)-\lambda_\ve) %-\frac{1}{\ve^{d-1}}\int_{\Omega}  J\Bigl(\frac{x-y}{\ve}\Bigr)\log\kappa(x,y)dy
 -\log \biggl( \frac{1}{\ve^d}\int_{\Omega}  J\Bigl(\frac{x-y}{\ve}\Bigr)\kappa \Bigl(x,y,\frac{x}{\ve},\frac{y}{\ve}\Bigr) dy\biggr).
 $$
 Let $x_\ve\in\overline{\Omega}$ be a maximum point of $W_\ve$. Choose a ball $B_{\tilde r_\ve}(\xi_\ve)$ contained in
 $B_{r_0 \ve}(x_\ve)\cap \Omega$. Then $\xi_\ve\in \Omega_\ve$, $B_{r_\ve}(\xi_\ve)\subset B_{r_0 \ve}(x_\ve)\cap \Omega$  and we have
\begin{equation}
\begin{aligned}
W_\ve(x_\ve)&\leq \frac{1}{\int_{B_{r_\ve}(\xi_\ve)} J\left(\frac{x_\ve-y}{\ve}\right)\kappa(x_\ve,y,\frac{x_\ve}{\ve},\frac{y}{\ve})dy}\int_{B_{r_\ve}(\xi_\ve)}  J\Bigl(\frac{x_\ve-y}{\ve}\Bigr)\kappa \Bigl(x_\ve,y,\frac{x_\ve}{\ve},\frac{y}{\ve}\Bigr) W_\ve(y)dy+C\ve\\
%&
%\leq \frac{1}{\int_{B_{\hat r_\ve}(\xi_\ve)} J\left(\frac{x-y}{\ve}\right)\kappa(x,y)dy}\int_{B_{\hat r_\ve}(\xi_\ve)}  J\Bigl(\frac{x-y}{\ve}\Bigr)\kappa(x,y) W^+_\ve(y)dy+C\ve \\
&\leq \frac{C_1}{\int_{B_{r_\ve}(\xi_\ve)} J\left(\frac{x_\ve-y}{\ve}\right)\kappa(x_\ve,y,\frac{x_\ve}{\ve},\frac{y}{\ve})dy}\int_{B_{ r_\ve}(\xi_\ve)}  W^+_\ve(y)dy+C\ve \leq C_2 \overline{W_\ve}^+(\xi_\ve)+C\ve\leq C_3.
\end{aligned}
 \end{equation}

\end{proof}

To show that $W^*$ is a subsolution  of \eqref{subsol_partial_lambda} consider an arbitrary test function $\Phi \in C_0^\infty(\mathbb{R}^d)$ and assume that
a $\max_{x\in \overline\Omega} (W(x)-\Phi(x))$ is attained at a point $x_0\in\Omega$, and this maximum is strict. Then we can extract a subsequence such that the maximum points $x_\ve$ of
$$
\Psi_\ve(x)= W_\ve(x)-\Phi(x)+\ve\log\varphi(x/\ve,\nabla \Phi(x),x)
$$
converge to $x_0$. We have
 $\Psi_\ve(x_\ve)- \Psi_\ve(y)\geq 0$ for $y\in\overline{\Omega}$, or
 $$
W_\ve(x_\ve)-W_\ve(y)\geq \Phi(x_\ve)-\Phi(y)+\ve\log\frac{\varphi(y/\ve,\nabla \Phi(x),x)}{\varphi(x_\ve/\ve,\nabla \Phi(x_\ve),x_\ve)},
 %\quad %\text{for} \
$$
therefore
%,
% and using the elementary inequality
%$e^\theta\geq 1+\theta$ we get
\begin{equation*}
%\label{nerivnist}
%\begin{aligned}
\int_{\Omega} E_\ve(x_\ve,y) %&
e^{-\frac{1}{\ve}(W_\ve(y)-W_\ve(x_\ve))}   dy\geq \int_{\Omega} E_\ve(x_\ve,y) e^{-\frac{1}{\ve}(\Phi(y)-\Phi(x_\ve))} \frac{\varphi(y/\ve,\nabla \Phi(y),y)} {\varphi(x_\ve/\ve,\nabla \Phi(x_\ve),x_\ve)}dy.
%\\
%&\quad+\frac{1}{\ve} \int_{\Omega} E_\ve(x_\ve,y) %J\Bigl(\frac{x_\ve-y}{\ve}\Bigr) \kappa\Bigl(x_\ve,y,\frac{x_\ve}{\ve}, \frac{y}{\ve}\Bigr)
%e^{-\frac{1}{\ve}(\Phi(y)-\Phi(x_\ve))} (\Psi_\ve(x_\ve)- \Psi_\ve(y)) \frac{\phi(y/\ve,\nabla \Phi(y),y)} {\phi(x_\ve/\ve,\nabla \Phi(x_\ve),x_\ve)}dy.
%\end{aligned}
\end{equation*}
% E_\ve(x,y)=\frac{1}{\ve^d}J\Bigl(\frac{x-y}{\ve}\Bigr)\kappa\Bigl(x,y,\frac{x}{\ve}, \frac{y}{\ve}\Bigr).
Then %Since the last term in \eqref{nerivnist} is nonnegative it follows that
\begin{equation*}
\begin{aligned}
- \lambda_\ve=&
\int_{\Omega} E_\ve(x_\ve,y) e^{-\frac{1}{\ve}(W_\ve(y)-W_\ve(x_\ve))}   dy -a\Bigl(x_\ve,\frac{x_\ve}{\ve}\Bigr)\\&
\geq\int_{\Omega} E_\ve(x_\ve,y)e^{-\frac{1}{\ve}(\Phi(y)-\Phi(x_\ve))} \frac{\varphi(y/\ve,\nabla \Phi(y),y)} {\varphi(x_\ve/\ve,\nabla \Phi(x_\ve),x_\ve)}dy -a\Bigl(x_\ve,\frac{x_\ve}{\ve}\Bigr),
\end{aligned}
\end{equation*}
and passing to the limit in this inequality as $\ve\to 0$ we derive $-H(\nabla \Phi(x_0),x_0)\leq -\lambda$. Thus
$W^*(x)$ is indeed an upper semicontinuous subsolution of  \eqref{subsol_partial_lambda}, being  a subsolution of \eqref{subsol_partial_lambda} function $W^\ast(x)$ is in fact Lipschitz continuous on $\Omega$ (see, e.g.,
Appendix A.3 in \cite{Ish2013}). Consequently  $-\lambda\geq \Lambda$.

Theorem \ref{localno_periodychna_teorema} is proved.
\end{proof}

%Now we are ready to prove Theorem \ref{Golovna_locperiodic}.

\noindent
{\it Proof of  Theorem \ref{Golovna_locperiodic}.}  For sufficiently small $\delta>0$ set
\begin{equation*}
\hat a(x)=\min\nolimits_{\xi\in \mathbb{T}^d} a(x,\xi),\quad \hat a^{(\delta)}(x)=\max\left\{ \hat a(x),\min\nolimits_{y\in\overline{\Omega}} \hat a(y) +\delta/2\right\},
\end{equation*}
%The assumptions made above about existence of principal eigenvalues will
and
$$
a^{(\delta)}(x,\xi)=\max \left\{a(x,\xi),\hat a^{(\delta)}(x)+\delta/2\right\}.
$$
Then for any $x\in\overline{\Omega}$ the function $a^{(\delta)}(x,\xi)$ attains its minimum over $\xi\in \mathbb{T}^d$ on a set of positive %($d$-dimensional Lebesgue)
measure, and
$a^{(\delta)}(x,\frac{x}{\ve})$ attains its minimum over $x\in\overline{\Omega}$ on a set of positive measure for sufficiently small $\ve$.  Hence we can
replace $a$ with $a^{(\delta)}$ to modify spectral problems \eqref{start_eq} and \eqref{typu_cell_pr-m} such that they do have some
principal eigenvalues $\lambda_\ve^{(\delta)}$ and $H^{(\delta)}(p,x)$ by Theorem 2.1 in \cite{LiCovWan2017}. Then applying Theorem
\ref{localno_periodychna_teorema} we get that $\lambda_\ve^{(\delta)}\to -\Lambda^{(\delta)}$ as $\ve\to 0$,  where
$\Lambda^{(\delta)}=\inf_{W\in C^1(\overline{\Omega})}\max_{x\in \overline\Omega} -H^{(\delta)}(\nabla W(x),x)$.
On the other hand $|a^{(\delta)}(x,\xi)-a(x,\xi)|\leq \delta$ and therefore $|\lambda_\ve^{(\delta)}-\lambda_\ve|\leq \delta$,
$|\Lambda-\Lambda^{(\delta)}|\leq \delta$. Thus, letting $\delta\to 0$ we obtain that $\lambda_\ve\to-\Lambda$,
 Theorem \ref{Golovna_locperiodic} is proved. \hfill $\qed$

\section{Acknowledgments}
The work of V. Rybalko
was partially supported by the SSF grant for Ukrainian scientists Dnr UKR22-0004.

%
% Writing
% $$
% \int_{\Omega}  J\Bigl(\frac{x-y}{\ve}\Bigr)\kappa(x,y) W_\ve(y)dy=\frac{1}{|B_1| r_\ve^{d}}\int_{B_{r_\ve}}\int_{\Omega-z}  J\Bigl(\frac{x-y-z}{\ve}\Bigr)\kappa(x,y+z) W_\ve(y+z)dydz\leq
% $$
% \end{proof}

\end{document}